\documentclass[10.5pt,colorinlistoftodos,disable]{article}%%%%% disable option is for removing the todo note list

\usepackage{bbm}%%%%%%%%%%%%%%%%%%%%%%%%  using the bbm fonts in math environment. bbm fonts: Black­board vari­ants of Com­puter Modern fonts
\usepackage{bm}%%%%%%%%%%%%%%%%%%%%%%%%%  This package defines commands to access bold math symbols.
\usepackage{eurosym}%%%%%%%%%%%%%%%%%%%%  The European currency symbol
\usepackage[T1]{fontenc}%%%%%%%%%%%%%%%  Solve the problem: Font shape `TU/ppl/m/sc' undefined
\usepackage[utf8]{inputenc}
\usepackage{color}

\usepackage{setspace}%%%%%%%%%%%%%%%%%%%  Set space between lines
\usepackage{fancyhdr}%%%%%%%%%%%%%%%%%%%  Extensive control of page headers and footers in LaTeX
\usepackage[leqno]{amsmath} %%% ???????
\usepackage{amsfonts,amssymb,amsthm,amscd,makeidx,mathrsfs,array,subeqnarray}
\usepackage{euscript} %%% Change the display of \mathcal alphabet
\usepackage{extarrows}%%% extended the arrows of amsmath
\allowdisplaybreaks%%%%%% \allowdisplaybreaks does not work in the environments of split,aligned,gathered, alignedat, etc.

\theoremstyle{plain}
\newtheorem{thm}{\bf Theorem}[section]
\newtheorem{coro}[thm]{\bf Corollary}
\newtheorem{lem}[thm]{\bf Lemma}
\newtheorem{prop}[thm]{\bf Proposition}
\newtheorem{defn}[thm]{\bf Definition}

\theoremstyle{remark}
\newtheorem{rmk}[thm]{\bf Remark}

\newtheorem{prob}{\bf Problem}[section]

\makeatletter \@addtoreset{equation}{section} \makeatother \makeatletter
\@addtoreset{thm}{section} \makeatother

\usepackage{graphicx}
\usepackage{caption}%%%%%%%%%%%%%%%%%%%%   Change the style of the graph caption
\usepackage{epstopdf}
\usepackage{enumerate}%%%%%%%%%%%%%%%%%%   The list environment

\usepackage{makeidx}
\usepackage[bookmarksnumbered, bookmarksopen,colorlinks,citecolor=blue,linkcolor=red,anchorcolor=blue,breaklinks]{hyperref}
\usepackage{chapterbib}%%%%%%%%%%%%%%%%%%% Multiple bibliographies in a document

\usepackage{lipsum}
\usepackage{todonotes}
\def\esssup{\mathop{\rm esssup}}

            \def\h{\widehat}          \def\wt{\widetilde}
             \def\cd{\cdot}            \def\cds{\cdots}

\def\({\Big (}              \def\){\Big )}
\def\[{\Big[}               \def\]{\Big]}

\def\Ra{\mathop{\Rightarrow}}                     
\def\lan{\mathop{\langle}}                        \def\ran{\mathop{\rangle}}
\def\llan{\left\langle}                           \def\rran{\right\rangle}
                             
\def\Blan{\Big\langle}                            \def\Bran{\Big\rangle}

        \def\dbE{\mathbb{E}}
\def\dbF{\mathbb{F}}    \def\dbH{\mathbb{H}}    
        
\def\dbP{\mathbb{P}}    \def\dbR{\mathbb{R}}  \def\dbS{\mathbb{S}}

\def\a{\alpha}    \def\g{\gamma}   \def\d{\delta}  \def\eps{\epsilon}  
\def\z{\zeta}     \def\l{\lambda}      \def\n{\nu}         \def\si{\sigma}
    \def\f{\varphi}            
\def\i{\infty}

    \def\G{\Gamma}   \def\D{\Delta}    \def\Th{\Theta}         \def\L{\Lambda}
    \def\F{\Phi}     \def\O{\Omega}

      \def\cD{{\cal D}}  
\def\cF{{\cal F}}        
        
    \def\cR{{\cal R}}  \def\cS{{\cal S}}  
\def\cU{{\cal U}}        
  \def\cl{{\cal l}}

      \def\ae{\hbox{\rm a.e.{ }}}             \def\as{\hbox{\rm a.s.{ }}}
              
      \def\tr{\hbox{\rm tr$\,$}}

\def\cl{\overline}       \def\deq{\triangleq}                   \def\les{\leqslant}                \def\ges{\geqslant}
     \def\={\buildrel \triangle \over =}    
\def\ds{\displaystyle}                                \def\ns{\noalign{\ss}}
\def\no{\noindent}                             \def\hb{\hbox}
\def\ss{\smallskip}      \def\ms{\medskip}                      
\def\5n{\negthinspace \negthinspace \negthinspace \negthinspace \negthinspace }
\def\4n{\negthinspace \negthinspace \negthinspace \negthinspace }
\def\3n{\negthinspace \negthinspace \negthinspace }
\def\2n{\negthinspace \negthinspace }
\def\1n{\negthinspace }

\def\q{\quad}            \def\qq{\qquad}

\font\tenbb=msbm10 \font\sevenbb=msbm7 \font\fivebb=msbm5
\newfam\bbfam
\scriptscriptfont\bbfam=\fivebb
\textfont\bbfam=\tenbb
\scriptfont\bbfam=\sevenbb

\def\bde{\begin{definition}}
  \def\ede{\end{definition}}
\def\be{\begin{equation}}
  \def\bel{\begin{equation}\label}
    \def\ee{\end{equation}}
  \def\bt{\begin{theorem}}
    \def\et{\end{theorem}}
  \def\bc{\begin{corollary}}
    \def\ec{\end{corollary}}
  \def\bl{\begin{lemma}}
    \def\el{\end{lemma}}
  \def\bp{\begin{proposition}}
    \def\ep{\end{proposition}}
  \def\bas{\begin{assumption}}
    \def\eas{\end{assumption}}
  \def\br{\begin{remark}}
    \def\er{\end{remark}}
  \def\ba{\begin{array}}
            \def\ea{\end{array}}
          \def\ed{\end{document}}

\begin{document}

\title{\bf Open-Loop and Closed-Loop Solvabilities for Stochastic Linear Quadratic Optimal Control
  Problems of Markov Regime-Switching System\thanks{This work is supported by the National Natural Science Foundation of China (grant nos. 11771079, 11371020), and RGC Grants 15209614 and 15255416.}}
\author{Xin Zhang\thanks{School of Mathematics, Southeast University, Nanjing, Jiangsu Province, 211189, China (x.zhang.seu@gmail.com).} \  and\  Xun Li\thanks{Department of
    Applied Mathematics, The Hong Kong Polytechnic University, Hong
    Kong, China (malixun@polyu.edu.hk).}
 }
\date{}
\maketitle

\noindent {\bf Abstract:}
This paper investigates the stochastic linear quadratic (LQ, for short) optimal control problem of Markov regime switching system. The representation of the cost functional for the stochastic LQ optimal control problem of Markov regime switching system is derived using the technique of It{\^o}'s formula. For the stochastic LQ optimal control problem of Markov regime switching system, we establish the equivalence between  the open-loop (closed-loop) solvability and the existence of an adapted solution to the corresponding forward-backward stochastic differential equation with constraint (the existence of a regular solution to the Riccati equation). Also, we analyze the interrelationship between the strongly regular solvability of the Riccati equation and the uniform convexity of the cost functional.    % the open-loop (closed-loop) solvability and the corresponding forward-backward differential equation system (the existence of a regular solution to the Riccati equation) for the stochastic LQ optimal control problem of Markov regime switching

\ms

\noindent {\bf Keywords:} linear quadratic optimal control, Markov regime switching, Riccati equation, open-loop solvability, closed-loop solvability.

\ms

\no\bf AMS Mathematics Subject Classification. \rm 49N10, 49N35, 93E20.

\listoftodos

\section{Introduction}
Linear-quadratic (LQ) optimal control problem plays important role in control theory. It is a classical and fundamental problem in the fields of control theory. In the past few decades, both the deterministic and stochastic linear quadratic (LQ) control problems are widely studied.
%The deterministic LQ optimal control problem can be trace back to the works of \citet{bellman1958samtcp} and \citet{kalman1960cto}. Extension to the
Stochastic LQ optimal control problem was first carried out by Kushner \cite{Kushner1962} with dynamic programming method. Later, Wonham \cite{Wonham1968}  studied the generalized version of the matrix Riccati equation arose in the problems of stochastic control and filtering. Using functional analysis techniques, Bismut \cite{Bismut1976} proved the existence of the Riccati equation and derived the existence of the optimal control in a random feedback form for stochastic LQ optimal control with random coefficients. Tang \cite{Tang2003} studied the existence and uniqueness of the associated stochastic Riccati equation for a general stochastic LQ optimal control problems with random coefficients and state control dependent noise via the method of stochastic flow, which solves Bismut and Peng's long-standing open problems. Moreover, Tang provided a rigorous derivation of the interrelationship between the Riccati equation and the stochastic Hamilton system as two different but equivalent tools for the stochastic LQ problem. For more details on the progress of stochastic Riccati equation, interest readers may refer to \cite{Kohlmann2003mbsr,Kohlmann2003mrlq,Kohlmann2002,Kohlmann2001ndbsre,Tang2015}.

Under some mild conditions on the weighting coefficients in the cost functional, such as positive definite of the quadratic weighting control martix, and so on,   the stochastic LQ optimal control problems can be solved elegantly via the Riccati equation approach, see \cite[Chapter 6]{yong1999sch}. Chen et al. \cite{Chen1998} was the first to start the pioneer work of stochastic LQ optimal control problems with indefinite of the quadratic weighting control matrix, which turns out to be useful in solving the continuous time mean-variance portfolio selection problems.  Since then, there has been an increasing interest in the so-called indefinite stochastic LQ optimal control, see, for example, Chen and Yong \cite{Chen2001}, Li and Zhou \cite{lizhou2002islq}, Li et al. \cite{xunli2001islqj,xunli2003islq}, and so on.
  %\citet{Lim1999}, \citet{Rami2000},  \citet{Chen2000}, \citet{Chen2001},  \citet{RamiMooreZhou2002islq}, \citet{xunli2003islq,xunli2001islqj, lizhou2002islq}, and so on.

Another extension to stochastic LQ optimal control problems is to involve random jumps in the state systems, such as Poisson jumps or the regime switching jumps. Wu and Wang \cite{wu2003fbsde} was the first to consider the stochastic LQ optimal control problems with Poisson jumps and obtain the existence and uniqueness of the deterministic Riccati equation. Using the technique of completing squares, Hu and Oksendal \cite{Hu2008} discussed  the stochastic LQ optimal control problem with Poisson jumps and partial information.
%\citet{Meng2014glqo} studied the multiple dimensional stochatic LQ optimal control problem with Poisson jumps and random coefficients.
Existence and uniqueness of the stochastic Riccati equation with jumps and connections between the stochastic Riccati equation with jumps  and the associated Hamilton systems of stochastic LQ optimal control problem were also presented. Yu \cite{Yu2017ihjd} investigated a kind of infinite horizon backward stochastic LQ optimal control problems and differential game problems under the jump-diffusion model state system. Li et al. \cite{Li2018} solved the indefinite  stochastic LQ optimal control problem with Poisson jumps.

The stochastic control problems involving regime switching jumps are of interest and of practical importance in various fields such as science, engineering, financial management and economics.  The regime-switching models and related topics have been extensively studied in the areas of applied probability and stochastic controls. More recently, there has been dramatically increasing interest in studying this family of stochastic control problems as well as their financial applications, see, for examples, \cite{Zhou2003mmvp,xunli2001islqj,Yin2004mmvps,lizhou2002islq,xunli2003islq,Zhang2018gsmp,Zhang2011mrsm,Zhang2012smp,Zhang2010psem,mei2017equilibrium}. Ji and Chizeck \cite{Ji1992jlqgc,Ji1990csctm} formulated a class of continuous-time LQ optimal controls with Markovian jumps. Zhang and Yin \cite{QingZhang1999noch} developed hybrid controls of a class of LQ systems modulated by a finite-state Markov chain. Li and Zhou \cite{lizhou2002islq}, Li et al. \cite{xunli2001islqj,xunli2003islq} introduced indefinite stochastic LQ optimal controls with regime switching jumps. Liu et al. \cite{Liu2005nocrs} considered near-optimal
controls of regime-switching LQ problems with indefinite control weight costs.

Recently, Sun and Yong \cite{sun2014linear} investigated the two-person zero-sum stochastic LQ differential games. It was shown in \cite{sun2014linear} that the open-loop solvability is equivalence to the existence of an adapted solution to an forward-backward stochastic differential equation (FBSDE, for short)  with constraint and closed loop solvability is equivalent to the existence of a regular solution to the Riccati equation. As a continuation work of \cite{sun2014linear}, Sun et al. \cite{Sun2016olcls} studied the open-loop and closed-loop solvabilities for stochastic LQ optimal control problems. Moreover, the equivalence between the strongly regular solvability of the Riccati equation and the uniform convexity of the cost functional is established. The aim of this paper is to extend the results of Sun et al. \cite{Sun2016olcls} to the case of stochastic LQ optimal control problems with regime switching jumps. We will establish the above equivalences of Sun et al. \cite{Sun2016olcls} for  % between the open-loop (closed-loop) solvability and the corresponding forward-backward differential equation system (the existence of a regular solution to the Riccati equation) for
the stochastic LQ optimal control problem with regime switching jumps.

The first main contribution of our paper is to provide a method for obtaining the representation of the cost functional for the stochastic LQ optimal control problem with regime switching jumps.  In Sun et al. \cite{Sun2016olcls}, the representation of the cost functional, which is the summary results of Yong and Zhou \cite{yong1999sch}, is fundamental to prove the above equivalences. Unlike the techniques of function analysis used in Yong and Zhou \cite{yong1999sch} or Sun et al. \cite{Sun2016olcls}, our method for deriving the representation of the cost functional is mainly based on the technique of It{\^o}'s formula only. The second main contribution of our paper is to use the stochastic flow theory for proving the equivalence between the closed-loop solvability and the existence of regular solution to the Riccati equation. Due to the incorporate of the regime switching jumps, the method used in Sun et al. \cite{Sun2016olcls} for proving the equivalence between the closed-loop solvability and the existence of regular solution to the Riccati equation does not work for the stochastic LQ optimal control problem with regime switching jumps.
%Inspired by \citet{Tang2003}, we filled this gap by using the technique of stochastic flow, which can also be used for proving the equivalence of \citet{Sun2016olcls}.

The rest of the paper is organized as follows. Section 2 will introduce some useful notations and collect some preliminary results and state the stochastic LQ optimal control problem with regime switching jumps. Section 3 is devoted to deriving the representation of the cost functional by using the technique of It\^o formula. In section 4 and 5, we will prove the equivalence between the open-loop (closed-loop) solvability and the existence of an adapted solution to the corresponding FBSDE with constraint (the existence of a regular solution to the Riccati equation) for the stochastic LQ optimal control problem of Markov regime switching system.  The equivalence between the strongly regular solvability of the Riccati equation and the uniform convexity of the cost functional is established in section 6.

\section{Preliminaries and Model Formulation}

Let $(\O,\cF,\dbF,\dbP)$ be a complete filtered probability space on
which a standard one-dimensional Brownian motion $W=\{W(t); 0\les t
< \i \}$ and a continuous time, finite-state, Markov chain $\a=\{\a(t); 0\les t< \i \}$ are defined, where $\dbF=\{\cF_t\}_{t\ges0}$ is the natural
filtration of $W$ and $\a$ augmented by all the $\dbP$-null sets in $\cF$.  %\citet{Karatzas-Shreve 1991,yong1999sch}.
In the rest of our paper, we will use the following notation.
\begin{eqnarray*}
  \begin{array}{ll}
    \mathbb{N}: & \mbox{the set of natural numbers};\\
    \dbR_+, \cl{\dbR}_+:     &\mbox{the sets } [0,\infty) \mbox{ and } [0,+\infty] \mbox{ respectively};\\
    \mathbb{R}^n: & \mbox{the } n\mbox{-dimensional Euclidean space};\\
    M^\top: & \mbox{the transpose of any vector or matrix } M;\\
    \tr[M]: & \mbox{the trace of a square matrix } M;\\
    \cR(M): & \mbox{the range of the matrix } M;\\
    \langle \cd\,,\cd\rangle:& \mbox{the inner products in possibly different Hilbert spaces};\\
    M^\dag:    & \mbox{the Moore-Penrose pseudo-inverse of the matrix } M  ({\rm see, \cite{penrose1955generalized}});\\
    \mathbb{R}^{n\times m}: & \mbox{the space of all } n\times m \mbox{ matrices endowed with the inner product }\\
                & \langle M, N\rangle \mapsto\tr[M^\top N] \mbox{ and the norm } |M|=\sqrt{\tr[M^\top M]};\\
    \dbS^n: & \mbox{the set of all }n\times n \mbox{ symmetric matrices};\\
    \cl{\dbS^n_+}: &\mbox{the set of all }n\times n \mbox{ positive semi-definite matrices};\\
    \dbS^n_+:  &\mbox{the set of all }n\times n \mbox{ positive-definite matrices}.\\
    % \mbox{Diag}(y): & \mbox{the diagonal matrix with the elements of } y \mbox{ on the diagonal } \\
    %            & \mbox{and } 0s \mbox{ everywhere else};\\%the diagonal matrix with the elements of y on the diagonal. and 0s everywhere else
   % \left < x, y \right >: & \mbox{the inner product of } x,y\in \mathbb{R}^L, \mbox{ that is }  \left < x, y \right >:=x^Ty;\\
    % C^{k,l}([0,T]\times \mathbb{R}^L), k,l\in \mathbb{N}: &  \mbox{the set of the functions } f(t,x) \mbox{  whose partial derivatives }\\
    %             & \mbox{of orders } \leq k \mbox{ in the first variable } t \mbox{ and of orders } \leq l \mbox{ in the }\\
    %             & \mbox{second variable } x  \mbox{ are continuous on } [0,T]\times \mathbb{R}^L.
  \end{array}
\end{eqnarray*}
% We recall that $\dbR^n$ is the $n$-dimensional Euclidean space,
% $\dbR^{n\times m}$ is the space of all $(n\times m)$ matrices,
% endowed with the inner product $\langle M,N\rangle\mapsto\tr[M^\top
% N]$ and the norm $|M|=\sqrt{\tr[M^\top M]}$, $\dbS^n\subseteq
% \dbR^{n\times n}$ is the set of all $(n\times n)$ symmetric
% matrices, $\cl{\dbS^n_+}\subseteq\dbS^n$ is the set of all $(n\times
% n)$  positive semi-definite matrices, and
% $\dbS^n_+\subseteq\cl{\dbS^n_+}$ is the set of all $(n\times n)$
% positive-definite matrices. When there is no confusion, we shall use
% $\langle \cd\,,\cd\rangle$ for inner products in possibly different
% Hilbert spaces. Also, $M^\dag$ stands for the (Moore-Penrose)
% pseudo-inverse of the matrix $M$ (\citet{penrose1955generalized}), and $\cR(M)$
% stands for the range of the matrix $M$.
Next, let $T>0$ be a fixed
time horizon. For any $t\in[0,T)$ and Euclidean space $\dbH$, let
$$\ba{ll}
C([t,T];\dbH)=\Big\{\f:[t,T]\to\dbH\bigm|\f(\cd)\hb{ is
  continuous }\1n\Big\},\\
\ns\ds
L^p(t,T;\dbH)=\left\{\f:[t,T]\to\dbH\biggm|\int_t^T|\f(s)|^pds<\i\right\},\q1\les p<\i,\\
\ns\ds
L^\infty(t,T;\dbH)=\left\{\f:[t,T]\to\dbH\biggm|\esssup_{s\in[t,T]}|\f(s)|<\i\right\}.\ea$$
We denote
$$\ba{ll}
\ns\ds L^2_{\cF_T}(\O;\dbH)=\Big\{\xi:\O\to\dbH\bigm|\xi\hb{ is
  $\cF_T$-measurable, }\dbE|\xi|^2<\i\Big\},\\
\ns\ds
L_\dbF^2(t,T;\dbH)=\left\{\f:[t,T]\times\O\to\dbH\bigm|\f(\cd)\hb{ is
    $\dbF$-progressively measurable},\dbE\int^T_t|\f(s)|^2ds<\i\right\},\\
\ns\ds
L_\dbF^2(\O;C([t,T];\dbH))=\left\{\f:[t,T]\times\O\to\dbH\bigm|\f(\cd)\hb{
    is $\dbF$-adapted, continuous, }\dbE\left[\sup_{s\in[t,T]}|\f(s)|^2\right]<\i\right\},\\
\ns\ds L^2_\dbF(\O;L^1(t,T;\dbH))=\left\{\f:[t,T]\times
  \O\to\dbH\bigm|\f(\cd)\hb{ is $\dbF$-progressively measurable},
  \dbE\left(\int_t^T|\f(s)|ds\right)^2<\i\right\}.\ea$$
For an $\dbS^n$-valued function $F(\cd)$ on $[t,T]$, we use the
notation $F(\cd)\gg0$ to indicate that $F(\cd)$ is uniformly
positive definite on $[t,T]$, i.e., there exists a constant
$\delta>0$ such that
$$F(s)\ges\delta I,\qq\ae~s\in[t,T].$$

Now we start to formulate our system. We identify the state space of the chain $\a$ with a finite set $S :=\{1, 2 \dots, D\}$, where $D\in \mathbb{N}$ and suppose that the chain is homogeneous and irreducible. To specify statistical or probabilistic properties of the chain $\a$, we define the generator
$\l(t) := [\l_{ij}(t)]_{i, j = 1, 2, \dots, D}$ of the chain under $\mathbb{P}$.
This is also called the rate matrix, or the $Q$-matrix. Here, for each $i, j = 1, 2, \dots, D$,
$\l_{ij}(t)$ is the constant transition intensity of the chain from state $i$ to state $j$
at time $t$. Note that $\l_{ij}(t) \ge 0$, for $i \neq j$ and
$\sum^{D}_{j = 1} \l_{ij}(t) = 0$, so $\l_{ii}(t) \le 0$. In what follows for each $i, j = 1,
2, \dots, D$ with $i \neq j$, we suppose that $\l_{ij}(t) > 0$, so
$\l_{ii}(t) < 0$. For each fixed $j = 1, 2, \cdots, D$, let $N_j(t)$ be the number of jumps into state $j$ up to time $t$ and set
$$\l_j (t) := \int_0^t\l_{\a(s-)\, j}I_{\{\a (s-)\neq j\}}ds=\sum^{D}_{i = 1, i \neq j}\int^{t}_{0}\l_{ij}(s)I_{\{\a(s-)=i\}} d s.$$
Following Elliott et al. \cite{elliott1994hmm},  we have that for each $j=1,2,\cdots, D$,
\begin{eqnarray}
  \label{eq:N}
  \widetilde{N}_j (t):=N_j(t)-\l_j(t)
\end{eqnarray}
is an $(\dbF, \dbP)$-martingale.

Consider the following
controlled Markov regime switching linear stochastic differential equation (SDE, for short)
on a finite horizon $[t,T]$:
\begin{equation}
  \label{state}
  \left\{  \begin{aligned}
    dX^u(s;t,x,i)&=\big[A(s,\a(s))X^u(s;t,x,i)+B(s,\a(s))u(s)+b(s,\a(s))\big]ds\\
      &\qq +\big[C(s,\a(s))X^u(s;t,x,i)+D(s,\a(s))u(s)+\si(s,\a(s))\big]dW(s),
      \qq s\in[t,T], \\
      X^u(t;t,x,i)&=x,\q \a(t)=i,
    \end{aligned}
    \right.
\end{equation}
where $A(\cd,\cd), B(\cd,\cd), C(\cd,\cd), D(\cd,\cd )$ are given deterministic
matrix-valued functions of proper dimensions, and $b(\cd,\cd), \si(\cd,\cd)$
are vector-valued $\dbF$-progressively measurable processes. In the
above, $X^u(\cd\,;t,x,i)$, valued in $\dbR^n$, is the {\it state process}, and
$u(\cd)$, valued in $\dbR^m$, is the {\it control process}.  Any $u(\cd)$ is called an {\it admissible control} on $[t,T]$, if it belongs to the  following Hilbert space:
$$\cU[t,T]=\left\{u:[t,T]\times\O\to\dbR^m\bigm|u(\cd)\hb{ is $\dbF$-progressively
measurable, }\dbE\int_t^T|u(s)|^2ds<\i\right\}.$$
%
% Any $u(\cd)\in\cU[t,T]$ is called an {\it admissible control} (on
% $[t,T]$). Under some mild conditions on the coefficients, for any
% {\it initial pair} $(t,x,i)\in[0,T)\times\dbR^n\times \cS$ and admissible
% control $u(\cd)\in\cU[t,T]$, (\ref{state}) admits a unique strong
% solution $X^u(\cd\, ;t,x,i)$.
For any admissible control $u(\cd)$, we consider the following general quadratic cost functional:
\begin{equation}
  \label{cost}
{\small  \begin{aligned}
   J(t,x,i;u(\cd))\deq&\dbE\Bigg\{\Blan G(T,\a(T))X^u(T;t,x,i)+2g(T,\a(T)),X^u(T;t,x,i)\Bran\\
    &\qq +\int_t^T\bigg[\Blan Q(s,\a(s))X^u(s;t,x,i)+2q(s,\a(s)), X^u(s;t,x,i)\Bran\\
    &\qq\qq\qq+2\Blan\1nS(s,\a(s))X^u(s;t,x,i),u(s)\Bran +\Blan R(s,\a(s))u(s)+2\rho(s,\a(s)),u(s)\Bran\bigg] ds\Bigg\},
     \end{aligned} }
\end{equation}
% \begin{equation}
%   \label{cost}
%   \ba{ll}
%   %
%   \ns\ds J(t,x,i;u(\cd))\deq\dbE\Bigg\{\lan G(T,\a(T))X^{x,u}(T),X^{x,u}(T)\ran+2\lan g(T,
%   \a(T)),X^{x,u}(T)\ran\\
%   %
%   \ns\ds\qq\qq\qq\qq\qq+\int_t^T\llan\begin{pmatrix}Q(s,\a(s))&\1nS(s,\a(s))^\top\\S(s,\a(s))&\1nR(s,\a(s))\end{pmatrix}
%     \begin{pmatrix}X^{x,u}(s)\\ u(s)\end{pmatrix},
%     \begin{pmatrix}X^{x,u}(s)\\u(s)\end{pmatrix}\rran ds\\
%     %
%     \ns\ds\qq\qq\qq\qq\qq+2\int_t^T\llan\begin{pmatrix}q(s,\a(s))\\ \rho(s,\a(s))\end{pmatrix},\begin{pmatrix}X^{x,u}(s)\\u(s)\end{pmatrix}\rran ds\Bigg\},\ea
% \end{equation}
%
where $G(T,i)$ is a symmetric matrix, $Q(\cd,i)$, $S(\cd,i)$, $R(\cd,i), i=1,\cdots,D$ are deterministic matrix-valued
functions of proper dimensions with $Q(\cd,i)^\top=Q(\cd,i)$, $R(\cd,i)^\top=R(\cd,i)$; $g(T,\cd)$ is allowed
to be an $\cF_T$-measurable random variable and $q(\cd,\cd), \rho(\cd,\cd)$
are allowed to be vector-valued $\dbF$-progressively measurable processes.

The following standard assumptions will be in force throughout this paper.

\ms

{\bf(H1)} The coefficients of the state equation satisfy the following: for each $i\in \cS$,
$$\left\{\2n\ba{ll}
  \ns\ds A(\cd,i)\in L^1(0,T;\dbR^{n\times n}),\q B(\cd,i)\in L^2(0,T;\dbR^{n\times m}), \q b(\cd,i)\in L^2_\dbF(\O;L^1(0,T;\dbR^n)),\\
  \ns\ds C(\cd,i)\in L^2(0,T;\dbR^{n\times n}),\q D(\cd,i)\in L^\i(0,T;\dbR^{n\times m}), \q\si(\cd,i)\in L_\dbF^2(0,T;\dbR^n).\ea\right.$$

{\bf(H2)} The weighting coefficients in the cost functional satisfy the following: for each $i\in \cS$
$$\left\{\2n\ba{ll}
  \ns\ds  G(T,i)\in\dbS^n,\q Q(\cd, i)\in L^1(0,T;\dbS^n),\q S(\cd, i)\in L^2(0,T;\dbR^{m\times
    n}),\q R(\cd, i)\in L^\infty(0,T;\dbS^m),\\
  \ns\ds g(T,i)\in L^2_{\cF_T}(\O;\dbR^n),\q q(\cd, i)\in L^2_\dbF(\O;L^1(0,T;\dbR^n)),\q\rho(\cd, i)\in
  L_\dbF^2(0,T;\dbR^m).\ea\right.$$
%
%\setreviewsoff

\ms

Now we sate the stochastic LQ optimal control problem for the Markov regime switching system as follows.
\begin{prob}
  \bf (M-SLQ) \rm For any given initial pair
  $(t,x,i)\in[0,T)\times\dbR^n\times \cS$, find a $u^*(\cd)\in\cU[t,T]$, such that
  \begin{equation}
    \label{optim}
    J(t,x,i;u^*(\cd))=\inf_{u(\cd)\in\cU[t,T]}J(t,x,i;u(\cd))\deq V(t,x,i).
  \end{equation}
\end{prob}
\noindent Any $u^*(\cd)\in\cU[t,T]$ satisfying
\eqref{optim} is called an {\it optimal control} of Problem (M-SLQ)
for the initial pair $(t,x,i)$, and the corresponding path $X^*(\cd)\equiv
X^{u^*}(\cd\,; t,x,i)$ is called an {\it optimal state process};
the pair $(X^*(\cd),u^*(\cd))$ is called an {\it optimal pair}. The
function $V(\cd\,,\cd\, , \cd)$ is called the {\it value function} of
Problem (M-SLQ). When $b(\cd,\cd), \si(\cd,\cd), g(T,\cd), q(\cd,\cd), \rho(\cd,\cd)=0$, we denote the corresponding Problem (M-SLQ) by Problem $\hb{(M-SLQ)}^0$.
The corresponding cost functional and value function are denoted by $J^0(t,x,i;u(\cd))$ and $V^0(t,x,i)$, respectively.

Similar to Sun et al. \cite{Sun2016olcls}, we introduce the following  definitions of open-loop (closed-loop) optimal control.

\begin{defn}\label{sec:defnofopen-closeloop}
% (i) Problem (M-SLQ) is said to be {\it finite
  %   at initial pair $(t,x,i)\in[0,T]\times\dbR^n\times \cS$} if%
  % \begin{eqnarray}
  %   \label{V>-infty}
  %   V(t,x,i)>-\i.
  % \end{eqnarray}
  % %
  % Problem (M-SLQ) is said to be {\it finite at $t\in[0,T]$} if
  % (\ref{V>-infty}) holds for all $(x,i)\in\dbR^n\times\cS$, and Problem (M-SLQ) is
  % said to be {\it finite} if (\ref{V>-infty}) holds for all
  % $(t,x,i)\in[0,T]\times\dbR^n\times \cS$.

  % \ms

  (i) An element $u^*(\cd)\in\cU[t,T]$ is called an {\it open-loop
    optimal control} of Problem (M-SLQ) for the initial pair
  $(t,x,i)\in[0,T]\times\dbR^n\times\cS$ if
  \begin{eqnarray}
    \label{open-opti}
    J(t,x,i;u^*(\cd))\les J(t,x,i;u(\cd)),\qq\forall
    u(\cd)\in\cU[t,T].
  \end{eqnarray}
  %
  % If an open-loop optimal control (uniquely) exists for
  % $(t,x,i)\in[0,T]\times\dbR^n\times\cS$, Problem (M-SLQ) is said to be ({\it
  %   uniquely}) {\it open-loop solvable at $(t,x,i)\in[0,T]\times\dbR^n\times\cS$}.
  % Problem (M-SLQ) is said to be ({\it uniquely}) {\it open-loop solvable
  %   at $t\in[0,T)$} if for the given $t$, (\ref{open-opti}) holds for
  % all $(x,i)\in\dbR^n\times\cS$, and Problem (M-SLQ) is said to be ({\it uniquely})
  % {\it open-loop solvable} (on $[0,T)\times\dbR^n\times\cS$) if it is
  % (uniquely) open-loop solvable at all $(t,x,i)\in[0,T)\times\dbR^n\times\cS$.

  (ii) A pair $(\Th^*(\cd),v^*(\cd))\in L^2(t,T;\dbR^{m\times
    n})\times\cU[t,T]$ is called a {\it closed-loop optimal strategy} of
  Problem (M-SLQ) on $[t,T]$ if
  \begin{eqnarray}
    \label{closed-opti}\ba{ll}
    \ns\ds J(t,x,i;\Th^*(\cd)X^*(\cd)+v^*(\cd))\les
    J(t,x,i;u(\cd)),\qq\forall (x,i)\in\dbR^n\times\cS,\q u(\cd)\in\cU[t,T],\ea
  \end{eqnarray}
  where $X^*(\cd)$ is the strong solution to the following closed-loop
  system:
  \begin{eqnarray}
    \label{closed-loop0}\left\{\2n\ba{ll}
    dX^*(s)=\Big\{\big[A(s,\a(s))+B(s,\a(s))\Th^*(s)\big]X^*(s)+B(s,\a(s))v^*(s)+b(s,\a(s))\Big\}ds\\
    \qq\qq+\Big\{\big[C(s,\a(s))+D(s,\a(s))\Th^*(s)\big]X^*(s)+D(s,\a(s))v^*(s)+\si(s,\a(s))\Big\}dW(s),
    % \qq s\in[t,T],
    \\
    X^*(t)=x.\ea\right.\hspace{-2cm}
  \end{eqnarray}
\end{defn}
% \bf Definition 2.1.
%
% If a closed-loop optimal strategy (uniquely) exists on $[t,T]$,
% Problem (M-SLQ) is said to be ({\it uniquely}) {\it closed-loop
% solvable on $[t,T]$}. Problem (M-SLQ) is said to be ({\it uniquely})
% {\it closed-loop solvable} if it is (uniquely) closed-loop solvable
% on any $[t,T]$.

% \ms

\rm

\begin{rmk}
  We emphasize that in the definition of closed-loop optimal strategy,
  (\ref{closed-opti}) has to be true for all $(x,i)\in\dbR^n\times \cS$. One sees
  that if $(\Th^*(\cd),v^*(\cd))$ is a closed-loop optimal strategy of
  problem (M-SLQ) on $[t,T]$, then the outcome
  $u^*(\cd)\equiv\Th^*(\cd)X^*(\cd)+v^*(\cd)$ is an open-loop optimal
  control of Problem (M-SLQ) for the initial pair $(t,X^*(t),\a(t))$. Hence,
  the existence of closed-loop optimal strategies implies the
  existence of open-loop optimal controls. But, the existence of
  open-loop optimal controls does not necessarily imply the existence
  of a closed-loop optimal strategy.
\end{rmk}

% Due to the above indicated situation, unlike in \citet{sun2014linear}, and in classical literature on LQ problems, we distinguish
% the notions of open-loop and closed-loop solvabilities for Problem
% (M-SLQ). We repeat here that for given initial time $t\in[0,T)$, an
% open-loop optimal control is allowed to depend on the initial state
% $x$, whereas, a closed-loop optimal strategy is required to be
% independent of the initial state $x$. Because of the nature of
% closed-loop strategies, we define the finiteness of Problem (M-SLQ)
% only in terms of open-loop controls.

To simply notation of our further analysis, we introduce the following forward-backward stochastic differential equation (FBSDE for short) on a finite horizon $[t,T]$:
\begin{equation}
  \label{generalstate}
  \left\{
    \begin{aligned}
      dX^u(s;t,x,i)=&\big[A(s,\a(s))X^u(s;t,x,i)+B(s,\a(s))u(s)+b(s,\a(s))\big]ds\\
      &+\big[C(s,\a(s))X^u(s;t,x,i)+D(s,\a(s))u(s)+\si(s,\a(s))\big]dW(s),\\
      dY^u(s;t,x,i)=&-\big[A(s,\a(s))^\top Y^u(s;t,x,i)+C(s,\a(s))^\top Z^u(s;t,x,i)\\
      &+Q(s,\a(s))X^u(s;t,x,i)+S(s,\a(s))^\top u(s)+q(s,\a(s))\big]ds\\
      &+Z^u(s;t,x,i)dW(s)+\sum_{k=1}^D\G_k^u(s;t,x,i)d\widetilde{N}_k(s)\qq s\in[t,T], \\
      X^u(t;t,x,i)=&x, \q \a(t)=i, \q  Y^u(T;t,x,i)=G(T,\a(T))X^u(T;t,x,i)+g(T,\a(T)).
    \end{aligned}\right.
\end{equation}
The solution of the above FBSDE system is denoted by $(X^u(\cd\,;t,x,i), Y^u(\cd\,;t,x,i), Z^u(\cd\,;t,x,i),\G^u(\cd\,;t,x,i))$, where $\G^u(\cd\,;t,x,i):=(\G_1^u(\cd\,;t,x,i),\cdots, \G_D^u(\cd\,;t,x,i))$. If the control $u(\cd)$ is chose as $\Th(\cd)X(\cd)+v(\cd)$, we will use the notation $$(X^{\Th,v}(\cd\,;t,x,i), Y^{\Th,v}(\cd\,;t,x,i), Z^{\Th,v}(\cd\,;t,x,i),\G^{\Th,v}(\cd\,;t,x,i))$$ denoting by the solution of the above FBSDE.  If $b(\cd,\cd)=\si(\cd,\cd)=q(\cd,\cd)=g(\cd,\cd)=0$, the solution of the above FBSDE is denoted by
\begin{equation*}
  (X_0^u(\cd\,;t,x,i), Y_0^u(\cd\, ;t,x,i), Z_0^u(\cd\, ;t,x,i),\G_0^u(\cd\, ;t,x,i)).
\end{equation*}

\ms

\ms

\section{Representation of the Cost Functional}

In this section, we will present a representation of the cost
functional for Problem (M-SLQ), which plays a crucial role in the study of
open-loop/closed-loop solvability of Problem (M-SLQ). Unlike the method used in Yong and Zhou \cite{yong1999sch}, we derive the representation of the cost functional using the technique of  It{\^o}'s formula.

% from which we will obtain some basic
% conditions ensuring the convexity of the cost functional. Convexity
% of the cost functional will play a crucial role in the study of
% finiteness and open-loop/closed-loop solvability of Problem (M-SLQ).
% The following proposition is a summary of some relevant results
% found in \citet{yong1999sch}.

\begin{prop}
  \label{RP-cost}
  \sl Let {\rm(H1)--(H2)} hold and $(X^u(\cd\,;t,x,i), Y^u(\cd\,;t,x,i), Z^u(\cd\,;t,x,i),\G^u(\cd\,;t,x,i))$ is the solution of \eqref{generalstate}. Then for $(x,i,u(\cd))\in\dbR^n\times\cS\times\cU[t,T]$,
  \begin{eqnarray}
    \label{J-rep1}
    \begin{aligned}
       J^0(t,x,i;u(\cd))&=\langle M_2(t,i)u,u\rangle+2\langle M_1(t,i)x,u\rangle+\langle M_0(t,i)x,x\rangle,\\
       J(t,x,i;u(\cd))&=\langle M_2(t,i)u,u\rangle+2\langle M_1(t,i)x,u\rangle+\langle M_0(t,i)x,x\rangle
      +2\langle \nu_t, u\rangle+2\langle y_t, x\rangle+c_t,
    \end{aligned}
    \end{eqnarray}
where
\begin{equation}
   \begin{aligned}
  %\label{L_0-rep}
  M_0(t,i)x=&\dbE[Y_0^0(t;t,x,i)],\\
  %\label{L_1-rep}
  (M_1(t,i)x)(s)=&B(s,\a(s))^\top Y_0^0(s;t,x,i)+D(s,\a(s))^\top Z_0^0(s;t,x,i)\\
  &+S(s,\a(s))X_0^0(s;t,x,i), \qq s\in[t,T],\\    %\label{L_2-rep}
  (M_2(t,i)u(\cd))(s)=&B(s,\a(s))^\top Y_0^u(s;t,0,i)+D(s,\a(s))^\top
  Z_0^u(s;t,0,i)\\
  &+S(s,\a(s))X_0^u(s;t,0,i)+R(s,\a(s))u(s), \qq s\in[t,T],
\end{aligned}
\end{equation}
and
\begin{equation}
  %\label{eq:9}
  \begin{aligned}
    y_t=&\dbE[Y^0(t;t,0,i)],\\
    v_t(s)=&\big[B(s,\a(s))]^\top Y^0(s;t,0,i)+D(s,\a(s))^\top Z^0(s;t,0,i)\\
    &\q+S(s,\a(s))X^0(s;t,0,i)+\rho(s,\a(s)),  \qq s\in[t,T],\\
    c_t=&\dbE\bigg[\llan G(T,\a(T))X^0(T;t,0,i)+2g(T,\a(T)),X^0(T;t,0,i)\rran\\
    &\q+\int_t^T\llan Q(s,\a(s))X^0(s;t,0,i)+2q(s,\a(s)),X^0(s;t,0,i)\rran ds\bigg].
  \end{aligned}
\end{equation}
\end{prop}
\begin{proof}
  Let
  \begin{equation*}
    \begin{aligned}
      I_1:=&\dbE\bigg[\Blan G(T,\a(T))X_0^u(T;t,x,i),X_0^u(T;t,x,i)\Bran\bigg],\\
      I_2:=&\dbE\bigg\{\int_t^T\bigg[\Blan Q(s,\a(s))X_0^u(s;t,x,i),X_0^u(s;t,x,i)\Bran\\
      &\qq+2\Blan S(s,\a(s))X_0^u(s;t,x,i), u(s)\Bran+\Blan R(s,\a(s))u(s),u(s)\Bran\bigg]ds\bigg\},
    \end{aligned}
  \end{equation*}
  % \begin{eqnarray*}
  %   I_1&:=&\dbE\bigg[\Blan G(T,\a(T))X_0^u(T;t,x,i),X_0^u(T;t,x,i)\Bran\bigg]\\
  %   I_2&:=&\dbE\bigg\{\int_t^T\bigg[\Blan Q(s,\a(s))X_0^u(s;t,x,i),X_0^u(s;t,x,i)\Bran+2\Blan S(s,\a(s))X_0^u(s;t,x,i), u(s)\Bran+\Blan R(s,\a(s))u(s),u(s)\Bran\bigg]ds\bigg\},
  % \end{eqnarray*}
  and we have
  \begin{eqnarray*}
    J^0(t,x,i;u(\cd))=I_1+I_2.
  \end{eqnarray*}
Observing that
\begin{eqnarray}
  X_0^u(\cd\,;t,x,i)=X_0^u(\cd\,;t,0,i)+X_0^0(\cd\,;t,x,i),
\end{eqnarray}
and therefore
\begin{equation*}
  \begin{aligned}
    I_1=&\dbE\bigg[\Blan G(T,\a(T))X_0^u(T;t,0,i),X_0^u(T;t,0,i)\Bran,\\
    &\qq\q+2\Blan G(T,\a(T))X_0^0(T;t,x,i),X_0^u(T;t,0,i)\Bran+\Blan G(T,\a(T))X_0^0(T;t,x,i),X_0^0(T;t,x,i)\Bran\bigg],\\
    I_2=&\dbE\bigg\{\int_t^T\bigg[\Blan Q(s,\a(s))X_0^u(s;t,0,i),X_0^u(s;t,0,i)\Bran+\Blan Q(s,\a(s))X_0^0(s;t,x,i), X_0^0(s;t,x,i)\Bran\\
    &\qq\q+2\Blan Q(s,\a(s))X_0^0(s;t,x,i), X_0^u(s;t,0,i)\Bran+2\Blan S(s,\a(s))X_0^u(s;t,0,i), u(s)\Bran\\
    &\qq\q+2\Blan S(s,\a(s))X_0^0(s;t,x,i),u(s)\Bran+\Blan R(s,\a(s))u(s),u(s)\Bran\bigg]ds\bigg\}.
  \end{aligned}
\end{equation*}
Applying It\^o's formula to $\langle Y_0^u(s;t,0,i),X_0^u(s;t,0,i)\rangle, \langle Y_0^0(s;t,x,i),X_0^u(s;t,0,i)\rangle$ and $\langle Y_0^0(s;t,x,i),X_0^0(s;t,x,i)\rangle$, we have
\begin{eqnarray*}
  J^0(t,x,i;u(\cd))&=&I_1+I_2\\
                   &=&\dbE\int_t^T\lan (M_2(t,i)u(\cd))(s), u(s)\ran ds+2\dbE\int_t^T\lan (M_1(t,i)x)(s), u(s)\ran ds +\lan\dbE[Y_0^0(t;t,x,i)],x\ran\\
                   &=&\langle M_2(t,i)u,u\rangle+2\langle M_1(t,i)x,u\rangle+\langle M_0(t,i)x,x\rangle.
\end{eqnarray*}
Let
\begin{eqnarray*}
  I_3&:=&\dbE\bigg[\Blan G(T,\a(T))X^u(T;t,x,i)+2g(T,\a(T)),X^u(T;t,x,i)\Bran\bigg],\\
  % +2\Blan g(T,\a(T)),X^{x,u}(T)\Bran\bigg]\\
  I_4&:=&\dbE\bigg\{\int_t^T\bigg[\Blan Q(s,\a(s))X^u(s;t,x,i)+2q(s,\a(s)),X^u(s;t,x,i)\Bran\\
     &&\qq\qq\q+2\Blan S(s,\a(s))X^u(s;t,x,i), u(s)\Bran+\Blan R(s,\a(s))u(s)+2\rho(s,\a(s)),u(s)\Bran\bigg]ds\bigg\},
\end{eqnarray*}
and we have
\begin{eqnarray*}
  J(t,x,i;u(\cd))=I_3+I_4.
\end{eqnarray*}
Observing that
\begin{eqnarray}
  X^u(\cd\,;t,x,i)=X_0^u(\cd\,;t,x,i)+X^0(\cd\,;t,0,i),
\end{eqnarray}
and therefore
\begin{eqnarray*}
  I_3=I_{31}+I_{32}+I_{33},\qq I_4=I_{41}+I_{42}+I_{43},
\end{eqnarray*}
where
\begin{eqnarray*}
  &&I_{31}:=\dbE\Blan G(T,\a(T))X_0^u(T;t,x,i),X_0^u(T;t,x,i)\Bran,\\
  &&I_{32}:=2\dbE \Blan G(T,\a(T))X^0(T;t,0,i)+g(T,\a(T)),X_0^u(T;t,x,i)\Bran,\\
  &&I_{33}:=\dbE\Blan G(T,\a(T))X^0(T;t,0,i)+2g(T,\a(T)),X^0(T;t,0,i),
\end{eqnarray*}
and
\begin{equation*}
  \begin{aligned}
    I_{41}&:=\dbE\int_t^T\bigg[\Blan Q(s,\a(s))X_0^u(s;t,x,i),X_0^u(s;t,x,i)\Bran\\
    &\qq\qq+2\Blan S(s,\a(s))X_0^u(s;t,x,i),u(s)\Bran+\Blan R(s,\a(s))u(s),u(s)\Bran\bigg]ds,\\
    I_{42}&:=2\dbE\int_t^T\bigg[\Blan Q(s,\a(s))X^0(s;t,0,i)+q(s,\a(s)), X_0^u(s;t,x,i)\Bran\\
    &\qq\qq+2\Blan S(s,\a(s))X^0(s;t,0,i)+\rho(s,\a(s)),u(s)\Bran\bigg]ds,\\
   I_{43}&:=\dbE\int_t^T\bigg[\Blan Q(s,\a(s))X^0(s;t,0,i)+2q(s,\a(s)), X^0(s;t,0,i)\Bran\bigg]ds.
  \end{aligned}
\end{equation*}
% \begin{eqnarray*}
%   && I_{41}:=\dbE\int_t^T\bigg[\Blan Q(s,\a(s))X_0^u(s;t,x,i),X_0^u(s;t,x,i)\Bran+2\Blan S(s,\a(s))X_0^u(s;t,x,i),u(s)\Bran+\Blan R(s,\a(s))u(s),u(s)\Bran\bigg]ds\\
%   &&I_{42}:=2\dbE\int_t^T\bigg[\Blan Q(s,\a(s))X^0(s;t,0,i)+q(s,\a(s)), X_0^u(s;t,x,i)\Bran+2\Blan S(s,\a(s))X^0(s;t,0,i)+\rho(s,\a(s)),u(s)\Bran\bigg]ds\\
%   &&I_{43}:=\dbE\int_t^T\bigg[\Blan Q(s,\a(s))X^0(s;t,0,i)+2q(s,\a(s)), X^0(s;t,0,i)\Bran\bigg]ds.
% \end{eqnarray*}
% \begin{eqnarray*}
%   &&I_3=\dbE\bigg[\Blan G(T,\a(T))X_0^{x,u}(T),X_0^{x,u}(T)\Bran\\
%      &&\qq\qq+2\Blan G(T,\a(T))X^{0,0}(T)+g(T,\a(T)),X_0^{x,u}(T)\Bran\\
%      &&\qq\qq+\Blan G(T,\a(T))X^{0,0}(T)+2g(T,\a(T)),X^{0,0}(T)\Bran\bigg],\\
%   &&I_4=\dbE\bigg\{\int_t^T\bigg[\Blan Q(s,\a(s))X_0^{x,u}(s),X_0^{x,u}(s)\Bran+2\Blan S(s,\a(s))X_0^{x,u}(s),u(s)\Bran+\Blan R(s,\a(s))u(s),u(s)\Bran\\
%      &&\qq\qq\qq+2\Blan Q(s,\a(s))X^{0,0}(s)+q(s,\a(s)), X_0^{x,u}(s)\Bran+2\Blan S(s,\a(s))X^{0,0}(s)+\rho(s,\a(s)),u(s)\Bran\\
%      &&\qq\qq\qq+\Blan Q(s,\a(s))X^{0,0}(s)+2q(s,\a(s)), X^{0,0}(s)\Bran\bigg]ds\bigg\}.
% \end{eqnarray*}
Applying It\^o's formula to $\lan Y^0(s;t,0,i), X_0^u(s;t,x,i)\ran$ yields
\begin{eqnarray*}
  I_{32}+I_{42}=2\lan \dbE Y^0(t;t,0,i),x\ran+2\dbE\int_t^T\lan v_t(s), u(s)\ran ds=2\lan y_t, x\ran+2\lan \nu_t, u\ran.
\end{eqnarray*}
%%% B(s,\a(s))^\top Y^{0,0}(s)+D(s,\a(s))^\top Z^{0,0}(s)+S(s,\a(s))X^{0,0}(s)+\rho(s,\a(s))
Noting that
\begin{eqnarray*}
  J^0(t,x,i;u(\cd))=I_{31}+I_{41},\qq c_t=I_{33}+I_{43}
\end{eqnarray*}
and therefore,
\begin{equation*}
  \begin{aligned}
    J(t,x,i;u(\cd))&= I_3+I_4=(I_{31}+I_{41})+(I_{32}+I_{42})+(I_{33}+I_{43})\\
   & = \langle M_2(t,i)u,u\rangle+2\langle M_1(t,i)x,u\rangle+\langle M_0(t,i)x,x\rangle
    +2\langle \nu_t, u\rangle+2\langle y_t, x\rangle+c_t.
  \end{aligned}
\end{equation*}
\end{proof}
Next we shall show that the above characterizes of operators $M_0(t,i)$ and $M_2(t,i)$  is equivalent to the results obtained by using the technique of function analysis.

\begin{prop}\label{sec:F-K} \sl
  $M_0(\cd,i)$ defined in \ref{RP-cost} admits the following Feynman-Kac representation:
  \begin{equation}
    \label{L_0}
      M_0(t,i)=\dbE\bigg[\F(T;t,i)^\top G(T,\a(T))\F(T;t,i)+\int_t^T\F(s;t,i)^\top Q(s,\a(s))\F(s;t,i)ds\bigg],
      \end{equation}
where $\F(\cd\,;t,i)$ is the solution to the following SDE for
  $\dbR^{n\times n}$-valued process:
  \begin{equation}
    \label{F}
    \left\{\begin{aligned}
        d\F(s;t,i)&=A(s,\a(s))\F(s;t,i)ds+C(s,\a(s))\F(s;t,i)dW(s),\qq s\in[t,T],\\
        \F(t;t,i)&=I, \q \a(t)=i.
    \end{aligned}
\right.
  \end{equation}
Furthermore, $M_0(t,i)$ also solves the following ordinary differential equations
\begin{equation}
  \label{3.8}
  \left\{\begin{aligned}
      \dot M_0(t,i)&+M_0(t,i)A(t,i)+A(t,i)^\top M_0(t,i)\\
                   &+C(t,i)^\top M_0(t,i)C(t,i)+Q(t,i)+\sum_{k=1}^D\lambda_{ik}(t)M_0(t,k)=0,\q (t,i)\in[0,T]\times\cS,\\
      M_0(T,i)&=G(T,i), \q i\in\cS.
  \end{aligned}
\right.
\end{equation}
  % \begin{eqnarray} \label{3.8}
  %  \hspace{0.8cm} \left\{\ba{l}
  %     %
  %   \dot M_0(t,i)+M_0(t,i)A(t,i)+A(t,i)^\top M_0(t,i)\\
  %   \hspace{1.25cm}+C(t,i)^\top
  %   M_0(t,i)C(t,i)+Q(t,i)+\sum_{k=1}^D\lambda_{ik}(t)M_0(t,k)=0,\q (t,i)\in[0,T]\times\cS,\\
  %   %
  %   M_0(T,i)=G(T,i), \q i\in\cS.\ea\right.
  % \end{eqnarray}
  % %
\end{prop}

\begin{proof}
  \rm Let $\F(\cd;t,i)$ be the solution to (\ref{F}). Then it is easy to verify that
  \begin{equation*}
    X_0^0(s;t,x,i)=\F(s;t,i)x.
  \end{equation*}
  % $\F(s)^{-1}$ exists for all $s\ges0$ and the following holds:
  % %
  % $$\left\{\2n\ba{ll}
  %   %
  %   \ns\ds d\big[\F(s)^{-1}\big]=-\F(s)^{-1}\big[A(s,\a(s))-C(s,\a(s))^2\big]ds-\F(s)^{-1}C(s,\a(s))dW(s),\qq s\ges0,\\
  %   %
  %   \ns\ds\F(0)^{-1}=I, \q \a(0)=i.\ea\right.$$
  % %
  % Applying It\^o's formula to $\F(s)^{-1}X_0^0(s;t,x,i)$ leads to
  % \begin{eqnarray*}
  %   \F(s)^{-1}X_0^0(s;t,x,i)=\F(t)^{-1}X_0^{x,0}(t)=\F(t)^{-1}(t)x, \q {\rm i.e.}\q  X_0^{x,0}=\F(s)\F(t)^{-1}x.
  % \end{eqnarray*}
Applying It\^o's formula to $\lan Y_0^0(s;t,x,i),X_0^0(s;t,x,i)\ran$, we can easily obtain
\begin{align*}
  &\dbE\big[\lan G(T,\a(T))X_0^0(T;t,x,i), X_0^0(T;t,x,i)\ran\big]\\
  &=\lan \dbE [Y_0^0(t;t,x,i)], x\ran-\dbE\bigg[\int_t^TX_0^0(s;t,x,i)^\top Q(s,\a(s))X_0^0(s;t,x,i)ds\bigg].
\end{align*}
Therefore,
\begin{equation*}
  \begin{aligned}
    \lan \dbE [Y_0^0(t;t,x,i)], x\ran&=\dbE\big[\lan G(T,\a(T))X_0^0(T;t,x,i), X_0^0(T;t,x,i)\ran\big]\\
    &\qq+\dbE\bigg[\int_t^TX_0^0(s;t,x,i)^\top Q(s,\a(s))X_0^0(s;t,x,i)ds\bigg]\\
    &=\dbE\big[\lan G(T,\a(T))\F(T;t,i)x, \F(T;t,i)x\ran\big]\\
    &\qq+\dbE\bigg[\int_t^Tx^\top\F(s;t,i)^\top Q(s,\a(s))\F(s;t,i)xds\bigg]\\
    &=\dbE\big[\lan \F(T;t,i)^\top G(T,\a(T))\F(T;t,i)x, x\ran\big]\\
    &\qq+\dbE\bigg[\int_t^T\lan\F(s;t,i)^\top Q(s,\a(s))\F(s;t,i)x,x\ran ds\bigg]\\
    &=\Blan\dbE\big[\F(T;t,i)^\top G(T,\a(T))\F(T;t,i)+\int_t^T\lan\F(s;t,i)^\top Q(s,\a(s))\F(s;t,i)ds\big]x, x\Bran.\\
  \end{aligned}
\end{equation*}
Thus observing that $M_0(t,i)x=\dbE\big[Y_0^0(t;t,x,i)\big]$, we have
\begin{eqnarray*}
  M_0(t,i)=\dbE\bigg[\F(T;t,i)^\top G(T,\a(T))\F(T;t,i)+\int_t^T\F(s;t,i)^\top Q(s,\a(s))\F(s;t,i)ds\bigg].
\end{eqnarray*}
Suppose $\wt M(\cd,i)$ satisfy the ODE \eqref{3.8}. Next we shall prove that $\wt M(\cd,i)=M_0(\cd,i)$. % From the above representation of $M_0(t,i)$, it is easy to verify $M_0(T,i)=G(T,i)$.
Observing that
\begin{eqnarray*}
  d\wt M(s,\a(s))=\dot{\wt{M}}(s,\a(s))ds+\sum_{k=1}^D\big[\wt M(s,k)-\wt M(s,\a(s-))\big]d\l_k(s)+\sum_{k=1}^D\big[\wt M(s,k)-\wt M(s,\a(s-))\big]d\wt N_k(s).
\end{eqnarray*}
Thus applying the It\^o's formula to $\F(s;t,i)^\top \wt M(s,\a(s))\F(s;t,i)$ leads to
\begin{eqnarray*}
\begin{array}{rl}
  \wt M(t,i)=\5n & \dbE\bigg[\F(T;t,i)^\top G(T,\a(T))\F(T;t,i)+\int_t^T\F(s;t,i)^\top Q(s,\a(s))\F(s;t,i)ds\bigg]\\
  =\5n & M_0(t,i).
\end{array}
\end{eqnarray*}
Thus we complete our proof.
\end{proof}

\begin{prop} \sl
  The operator $M_2(\cd,i)$ defined in \ref{RP-cost} admits the following representation:
  \begin{eqnarray}
    M_2(t,i)=\h L_t^*G(T,\a(T))\h L_t+ L_t^*Q(\cd,\a(\cd))L_t + S(\cd,\a(\cd))L_t+ L_t^*S(\cd,\a(\cd))^\top +R(\cd,\a(\cd)),
  \end{eqnarray}
  where the operators
  \begin{eqnarray}
    L_t:\cU[t,T]\rightarrow L_\dbF^2(t,T;\dbR^n), \qq \h L_t:\cU[t,T]\rightarrow L_{\dbF_T}^2(\O;\dbR^n)
  \end{eqnarray}
are defined as follows:
\begin{align}
  (L_tu)(\cd)&=\F(\cd\,;t,i)\bigg\{\int_t^\cdot \F(r;t,i)^{-1}\big[B(r,\a(r))-C(r,\a(r))D(r,\a(r))\big]u(r)dr\\
             &\qq\q\qq\qq+\int_t^\cdot \F(r;t,i)^{-1}D(r,\a(r))u(r)dW(r)\bigg\},\nonumber\\
  \h L_tu&=(L_tu)(T),
\end{align}
and $L_t^*$ and $\h L_t^*$ are the adjoint operators of $L_t$ and $\h L_t$, respectively.
\end{prop}

\begin{proof}
  Noting that the solution $X_0^u(\cd;t,0,i)$ of \eqref{generalstate} can be written as follows:
  \begin{eqnarray}
    X_0^u(s;t,0,i)&\5n=\5n&\F(s;t,i)\bigg\{\int_t^s\F(r;t,i)^{-1}\big[B(r,\a(r))-C(r,\a(r))D(r,\a(r))\big]u(r)dr\\
                &&\qq\qq+\int_t^s\F(r;t,i)^{-1}D(r,\a(r))u(r)dW(r)\bigg\}\nonumber\\
                &\5n=\5n&(L_tu)(s).\nonumber
  \end{eqnarray}
Applying It\^o's formula to $\lan Y_0^u(s;t,0,i),X_0^u(s;t,0,i)\ran$ yields
\begin{eqnarray*}
  \lan (M_2(t,i))u,u\ran &\5n=\5n&\dbE\bigg\{\lan G(T,\a(T))X_0^u(T;t,0,i), X_0^u(T;t,0,i)\ran\\
                         &&\qq+\int_t^T\bigg[\lan Q(s,\a(s))X_0^u(s;t,0,i), X_0^u(s;t,0,i)\ran+\lan S(s,\a(s))X_0^u(s;t,0,i),u(s)\ran\\
                         &&\qq\qq\qq+\lan S(s,\a(s))^\top u(s), X_0^u(s;t,0,i)\ran+\lan R(s,\a(s))u(s),u(s)\ran \bigg]ds\bigg\}\\
                         &\5n=\5n&\dbE\big[\lan G(T,\a(T))\h L_t u, \h L_t u\ran\big]+\lan Q(\cd,\a(\cd))L_tu, L_tu\ran\\
                         &&\qq+\lan S(\cd,\a(\cd))L_tu,u\ran+\lan S(\cd,\a(\cd))^\top u, L_tu\ran+\lan R(\cd,\a(\cd))u,u\ran\\
                         &\5n=\5n&\Blan \big[\h L_t^*G(T,\a(T))\h L_t+ L_t^*Q(\cd,\a(\cd))L_t + S(\cd,\a(\cd))L_t+ L_t^*S(\cd,\a(\cd))^\top +R(\cd,\a(\cd))\big]u,u\Bran.
\end{eqnarray*}
Thus we complete the proof.%\todo{We may also give explicit expression of the operator $L_t^*$ and $\h L_t^*$}
\end{proof}

From the representation of the cost functional, we have the
following simple corollary.

\ms
\begin{coro}\label{sec:frechetdifferential}
  \sl Let {\rm(H1)--(H2)} hold and $t\in[0,T)$ be
  given. For any $x\in\dbR^n, \eps\in\dbR$ and $u(\cd),
  v(\cd)\in\cU[t,T]$, the following holds:
  \bel{u+v-1}\ba{ll}
  \ns\ds J(t,x,i;u(\cd)+\eps v(\cd))=J(t,x,i;u(\cd))+\eps^2J^0(t,0,i;v(\cd))+2\eps\dbE\int_t^T\lan \bar{M}(t,i)(x,u)(s),v(s)\ran ds,\ea\ee
  where
  \begin{eqnarray}\label{eq:barM}
    \begin{aligned}
      \bar{M}(t,i)(x,u)(s):=\,&B(s,\a(s))^\top Y^u(s;t,x,i)\1n+\1nD(s,\a(s))^\top
      Z^u(s;t,x,i)\\
      &+S(s,\a(s))X^u(s;t,x,i)\1n +\1nR(s,\a(s))u(s)\1n+\1n\rho(s,\a(s)), \q s\in[t, T].
    \end{aligned}
     \end{eqnarray}
  % $(X(\cd),Y(\cd),Z(\cd))$ is the adapted solution to the
  % following (decoupled) linear FBSDE:
  % %
  % \bel{FBSDE03}\left\{\2n\ba{ll}
  %   %
  %   \ns\ds
  %   dX(s)=\big[A(s)X(s)+B(s)u(s)+b(s)\big]ds\\
  %   %
  %   \ns\ds\qq\qq~+\big[C(s)X(s)+D(s)u(s)+\si(s)\big]dW(s),\qq
  %   s\in[t,T],\\
  %   %
  %   \ns\ds dY(s)=-\big[A(s)^\top Y(s)+C(s)^\top Z(s)+Q(s)X(s)+S(s)^\top u(s)+q(s)\big]ds\\
  %   %
  %   \ns\ds\qq\qq\qq\qq\qq\qq\qq\qq\q~+Z(s)dW(s),\qq s\in[t,T], \\
  %   %
  %   \ns\ds X(t)=x,\qq Y(T)=GX(T)+g.\ea\right.\ee
  %
  Consequently, the map $u(\cd)\mapsto J(t,x,i;u(\cd))$ is Fr\'echet
  differentiable with the Fr\'echet derivative given by
  \begin{eqnarray}
    \label{DJ}
    \cD J(t,x,i;u(\cd))(s)=2\bar{M}(t,i)(x,u)(s),\qq s\in[t,T],
  \end{eqnarray}
  and (\ref{u+v-1}) can also be written as
  \begin{eqnarray}
    \label{u+v-1*}J(t,x,i;u(\cd)+\eps
    v(\cd))=J(t,x,i;u(\cd))+\eps^2J^0(t,0,i;v(\cd))+\eps\dbE\int_t^T\lan\cD
    J(t,x,i;u(\cd))(s),v(s)\ran ds.\hspace{-0.8cm}
  \end{eqnarray}
\end{coro}
% \bf Corollary 3.2.
% \ms

\begin{proof}
  \rm From Proposition \ref{RP-cost}, we have
  $$\ba{ll}
  \ns\ds J(t,x,i;u(\cd)+\eps v(\cd))\\
  \ns\ds=\lan M_2(t,i)(u+\eps v),u+\eps v\ran+2\lan M_1(t,i)x,u+\eps v\ran+\lan M_0(t,i)x,x\ran
  +2\lan \nu_t, u+\eps v\ran+2\lan y_t, x\ran+c_t\\
  \ns\ds=\lan M_2(t,i)u,u\ran+2\eps\lan M_2(t,i)u,v\ran+\eps^2\lan
  M_2(t,i)v,v\ran+2\lan M_1(t,i)x,u\ran
  +2\eps\lan M_1(t,i)x,v\ran+\lan M_0(t,i)x,x\ran\\
  \ns\ds\q~+2\lan \nu_t, u\ran+2\eps\lan \nu_t, v\ran+2\lan y_t, x\ran+c_t\\
  \ns\ds=J(t,x,i;u(\cd))+\eps^2J^0(t,0;v(\cd))+2\eps\lan M_2(t,i)u+M_1(t,i)x+\nu_t,v\ran.\ea$$
  From the representation of $M_1(t,i)$, $M_2(t,i)$ and $\nu_t$ in Proposition \ref{RP-cost} and the fact
  \begin{equation*}
    X^u(\cd\,;t,x,i)=X_0^u(\cd\,;t,x,i)+X^0(\cd\,;t,0,i),
  \end{equation*}
  we see that
  \begin{eqnarray*}
    (M_2(t,i)u)(s)+(M_1(t,i)x)(s)+\nu_t(s)&\5n=\5n&B(s,\a(s))^\top Y^u(s;t,x,i)+D(s,\a(s))^\top Z^u(s;t,x,i)\\
                                          &&+S(s,\a(s))X^u(s;t,x,i)+R(s,\a(s))u(s)+\rho(s,\a(s))\\
                                          &\5n=\5n&\bar{M}(t,i)(x,u)(s),\q s\in[t,T].
  \end{eqnarray*}
\end{proof}

%\section{Equivalences for open-loop and  closed-loop solvabilities}
\section{Open-loop Solvabilities}
We first present the equivalence between the open-loop solvability and the corresponding forward-backward differential equation system.

% between the open-loop (closed-loop) solvability and the corresponding forward-backward differential equation system (the existence of a regular solution to the Riccati equation)

\begin{thm} \sl
Let {\rm(H1)--(H2)} hold and $(t,x,i)\in [t,T]\times \dbR^n\times \cS$ be given. An element $u(\cd)\in\cU[t,T]$ is an open-loop
  optimal control of Problem {\rm(M-SLQ)} if and only if $J^0(t,0,i;v(\cd))\ge 0, \forall v(\cd)\in \cU[t,T]$ and the following stationary condition hold:
  \begin{equation}
    \label{J_u=0}
    \begin{aligned}
      &B(s,\a(s))^\top Y^{u}(s;t,x,i)+D(s,\a(s))^\top
      Z^{u}(s;t,x,i)\\
      &+S(s,\a(s))X^{u}(s;t,x,i)+R(s,\a(s))u(s)+\rho(s,\a(s))=0,\qq  s\in[t,T],
    \end{aligned}
  \end{equation}
  where $(X^{u}(\cd\,;t,x,i), Y^{u}(\cd\,;t,x,i), Z^{u}(\cd\,;t,x,i))$ is the adapted solution to the FBSDE \eqref{generalstate}.
\end{thm}
\begin{proof}
  By definition, $u(\cd)$ is an open-loop optimal control if and only if  the following hold:
  \begin{equation}
    \label{eq:oloopopticondi}
    J(t,x,i; u(\cd)+\eps v(\cd))-J(t,x,i;u)\ges 0,\q  \forall v(\cd)\in \cU[t,T].
  \end{equation}
While from Corollary \ref{sec:frechetdifferential}, we have
\begin{equation*}
  \begin{aligned}
    J(t,x,i; u(\cd)+\eps v(\cd))- J(t,x,i;u)=\eps^2J^0(t,0,i;v(\cd))+2\eps\dbE\int_t^T\lan \bar{M}(t,i)(x,u)(s),v(s)\ran ds.
  \end{aligned}
\end{equation*}
Therefore, \eqref{eq:oloopopticondi} holds if and only if $J^0(t,0,i;v(\cd))\ges 0, \forall v(\cd)\in \cU[t,T]$ and $\bar{M}(t,i)(x,u)(s)=0, s\in[t,T]$.
Note the definition of $\bar{M}$ in \eqref{eq:barM} and so the proof is completed.\end{proof}
\begin{rmk}
  Note that if $u(\cd)$ happens to be an open-loop optimal control of
  Problem (M-SLQ), then the  {\it stationarity condition} \eqref{J_u=0}
  holds,
  % %
  % \begin{eqnarray}
  %   \label{J_u=0}
  %   \cD J(t,x;u(\cd))&=&2\big[B(s,\a(s))^\top Y^u(s;t,x,i)+D(s,\a(s))^\top
  %                        Z^u(s;t,x,i)\\
  %                    &&\q+S(s,\a(s))X^u(s;t,x,i)+R(s,\a(s))u(s)+\rho(s,\a(s))\big]=0,\qq s\in[t,T],\nonumber
  % \end{eqnarray}
  % %
  which brings a coupling into the FBSDE \eqref{generalstate}. We call
  \eqref{generalstate}, together with the stationarity condition
  (\ref{J_u=0}), the {\it optimality system} for the open-loop optimal
  control of Problem (M-SLQ).
\end{rmk}

Next we shall investigate the relationships between open-loop solvability and  uniform convexity of the cost functional. We first introduce the definition of uniform convexity, which is from Zalinescu \cite[page 203]{Zalinescu2002cagvs} or \cite{Zalinescu1983ucf}.

\begin{defn}
  For a general normed space $(\dbH, \lVert\cd\rVert )$, the function $f:(\dbH, \lVert\cd\rVert)\mapsto \cl{\dbR}$ is said to be uniformly convex if there exists $h:\dbR_+\mapsto \cl{\dbR}_+$ with $h(t)>0$ for $t>0$ and $h(0)=0$ such that
  \begin{align*}
    f(\eps x+(1-\l)y)\les \eps f(x)+(1-\eps)f(y)-\eps(1-\eps)h(\rVert x-y\rVert),\  \forall x, y\in \mbox{\rm dom} f, \, \eps\in [0,1].
  \end{align*}
  % for all $x,y\in\mbox{\rm dom } f$ and $\l\in [0,1]$
  % where $h$ there exists $h:\dbR_+\mapsto \cl{\dbR}_+$ with $h(0)=0$ such that
\end{defn}

\begin{prop} \sl
  The cost functional $J(t,x,i;u(\cd))$ \todo{How to obtain the uniform convexity of $J(t,x,i,u(\cd))$ from the uniform convexity  of $J^0(t,0,i;u(\cd))$}
  is uniformly convex if and only if $M_2(t,i)\ges\eps I$ for some $\eps>0$, which is also equivalent to
  \begin{align}
   \label{J>l}
    J^0(t,0,i;u(\cd)) \ges \eps \dbE\int_t^T|u(s)|^2ds,\qq\forall u(\cd)\in\cU[t,T],
  \end{align}
  for some $\eps>0$.
\end{prop}
\begin{proof}
  From Proposition \ref{RP-cost}, we can see that for any $u(\cd), v(\cd)\in \cU[t,T]$ and $\eps\in[0,1]$,
  \begin{equation*}
    \begin{aligned}
      &J(t,x,i;\eps u(\cd)+(1-\eps)v(\cd))\\
      &=\langle M_2(t,i)(\eps u+(1-\eps)v,\eps u+(1-\eps)v\rangle+2\langle M_1(t,i)x,\eps u+(1-\eps)v\rangle\\
      &\q+\langle M_0(t,i)x,x\rangle +2\langle \nu_t, \eps u+(1-\eps)v\rangle+2\langle y_t, x\rangle+c_t\\
      &=\eps \big[\langle M_2(t,i)u,u\rangle+2\langle M_1(t,i)x,u\rangle+\langle M_0(t,i)x,x\rangle
      +2\langle \nu_t, u\rangle+2\langle y_t, x\rangle+c_t\big]\\
      &\q+(1-\eps)\big[\langle M_2(t,i)v,v\rangle+2\langle M_1(t,i)x,v\rangle+\langle M_0(t,i)x,x\rangle
      +2\langle \nu_t, v\rangle+2\langle y_t, x\rangle+c_t\big]\\
      &\q-\eps(1-\eps)\lan M_2(t,i)(u-v),u-v\ran.
    \end{aligned}
  \end{equation*}
  Thus from the definition of uniformly convex, the cost functional $J(t,x,i;u(\cd))$ is uniformly convex if and only if there exists $h:\dbR_+\mapsto \cl{\dbR}_+$ with $h(t)>0$ for $t>0$ and $h(0)=0$ such that
  \begin{equation*}
    \lan M_2(t,i)(u-v),u-v\ran\ges h(\lVert u-v\rVert),
  \end{equation*}
  which equivalent to $M_2(t,i)\ges \eps I$ for some $\eps>0$. From Proposition \ref{RP-cost}, we have
  \begin{align*}
    J^0(t,0,i;u(\cd))&=\langle M_2(t,i)u,u\rangle.
  \end{align*}
Therefore, $M_2(t,i)>\eps I$ for some $\eps>0$ if and only if
\begin{align*}
  J^0(t,0,i;u(\cd)) \ges \eps \dbE\int_t^T|u(s)|^2ds,\qq\forall u(\cd)\in\cU[t,T].
\end{align*}
Thus the proof is completed.
\end{proof}
\begin{rmk}
  From the definition of uniform convexity, one can easily verify that $J^0(t,x,i;u(\cd))$ is uniformly convex if and only if \eqref{J>l} is satisfied. So the uniform convexity of $J(t,x,i;u(\cd))$ is equivalent to the uniform convexity of $J^0(t,x,i;u(\cd))$.
\end{rmk}

It is obvious that if the following standard
  conditions
  \begin{eqnarray}
    \label{classical}G(T,i)\ges0,\q R(s,i)\ges\d I,\q Q(s,i)-S(s,i)^\top
    R(s,i)^{-1}S(s,i)\ges0,\q i\in \cS, \q \ae~s\in[0,T],
  \end{eqnarray}
  hold for some $\d>0$, then
  \begin{eqnarray*}
    M_2(t,i)&=&\h L_t^*G(T,\a(T))\h L_t+L_t^*\big[Q(\cd,\a(\cd))-S(\cd,\a(\cd))^\top R(\cd,\a(\cd))^{-1}S(\cd,\a(\cd))\big]L_t\\
            &&+\big[L_t^*S(\cd,\a(\cd))^\top
               R(\cd,\a(\cd))^{-{1\over2}}+R(\cd,\a(\cd))^{1\over2}\big]\big[R(\cd,\a(\cd))^{-{1\over2}}S(\cd,\a(\cd)L_t+R(\cd,\a(\cd))^{1\over2}\big]\\
            &\ges& 0,
  \end{eqnarray*}
  which means that the functional $u(\cd)\mapsto J^0(t,0,i;u(\cd))$ is
  convex. In fact, one actually has the uniform convexity of the cost functional $J^0(t,0,i;u(\cd))$ under standard conditions (\ref{classical}).  We  first present a lemma for proving the uniform convexity of $J^0(t,x,i;u(\cd))$.
  \begin{lem}
    \label{uniformconvex}
    \sl For any $u(\cd)\in\cU[t,T]$, let $X_0^u(\cd\,;t,0,i)$
    be the solution of \eqref{generalstate} with $x=0, b(\cd,\cd)=\si(\cd,\cd)=0.$
    %
    % \bel{}\left\{\2n\ba{ll}
    %   %
    %   \ns\ds dX^{(u)}(s)=\big[A(s)X^{(u)}(s)+B(s)u(s)\big]ds+\big[C(s)X^{(u)}(s)+D(s)u(s)\big]dW(s),\qq s\in[t,T], \\
    %   %
    %   \ns\ds X^{(u)}(t)=0.\ea\right.\ee
    % %
    Then for any $\Th(\cd,i)\in L^2(t,T;\dbR^{m\times n}), i\in\cS$, there exists a constant $\gamma>0$ such that
    \bel{lem-2.6}\dbE\int_t^T\big|u(s)-\Th(s) X_0^u(s;t,0,i)\big|^2ds
    \ges\g\dbE\int_t^T|u(s)|^2ds,\qq\forall u(\cd)\in\cU[t,T].\ee
  \end{lem}
  % \ms
  \begin{proof}
The proof is similar to Lemma 2.3 of Sun et al. \cite{Sun2016olcls} and so we omit it here.
  \end{proof}
  % \begin{proof}
  %   \rm Let $\Th(\cd,i)\in L^2(t,T;\dbR^{m\times n}), i\in\cS$. Define a bounded linear operator $\mathfrak{L}:\cU[t,T]\to\cU[t,T]$ by
  %   %
  %   $$\mathfrak{L}u(\cd)=u(\cd)-\Th(\cd) X_0^u(\cd\,;t,0,i).$$
  %   % $$\mathfrak{L}u(\cd)=u(\cd)-\Th(\cd) X^{0,u}(\cd).$$
  %   %
  %   Then $\mathfrak{L}$ is bijective and its inverse $\mathfrak{L}^{-1}$ is given by
  %   %
  %   $$\mathfrak{L}^{-1}u(\cd)=u(\cd)+\Th(\cd)\wt X_0^u(\cd\,;t,0,i),$$
  %   %
  %   where $\wt X_0^u(\cd\,;t,0,i)$ is the solution of
  %   %
  %   \begin{equation*}
  %     \left\{\begin{aligned}
  %         d\wt X_0^u(s;t,0,i)=&\Big\{\big[A(s,\a(s))+B(s,\a(s))\Th(s)\big]\wt X_0^u(s;t,0,i)+B(s,\a(s))u(s)\Big\}ds\\
  %         &+\Big\{\big[C(s,\a(s))+D(s,\a(s))\Th(s)\big]\wt X_0^u(s;t,0,i)+D(s,\a(s))u(s)\Big\}dW(s),\q s\in[t,T], \\
  %         \wt X_0^u(t;t,0,i)=&0,\q\a(t)=i.      %
  %       \end{aligned}\right.
  %   \end{equation*}
  %   By the bounded inverse theorem, $\mathfrak{L}^{-1}$ is bounded with
  %   norm $\|\mathfrak{L}^{-1}\|>0$. Thus,
  %   %
  %   $$\ba{ll}
  %   %
  %   \ns\ds\dbE\int_t^T|u(s)|^2ds=\dbE\int_t^T|(\mathfrak{L}^{-1}\mathfrak{L}u)(s)|^2ds
  %   \les\|\mathfrak{L}^{-1}\|\dbE\int_t^T|(\mathfrak{L}u)(s)|^2ds\\
  %   %
  %   \ns\ds\qq\qq\qq~~\1n=\|\mathfrak{L}^{-1}\|\dbE\int_t^T\big|u(s)-\Th(s) X_0^u(s;t,0,i)\big|^2ds,
  %   \qq\forall u(\cd)\in\cU[t,T],\ea$$
  %   %
  %   which implies (\ref{lem-2.6}) with $\g=\|\mathfrak{L}^{-1}\|^{-1}$.
  % \end{proof}

\begin{prop}
  \sl Let {\rm(H1)--(H2)} and {\rm(\ref{classical})} hold. Then for any $(t,i)\in[0,T)\times\cS$, the
  map $u(\cd)\mapsto J^0(t,0,i;u(\cd))$ is uniformly convex.
\end{prop}

\begin{proof}
  By Lemma \ref{uniformconvex} (taking $\Th(\cd)=-R(\cd,\cd)^{-1}S(\cd,\cd)$), we have
  \begin{eqnarray*}
    J^0(t,0,i;u(\cd))
    &=&\dbE\bigg\{\langle G(T,\a(T))X_0^u(T;t,0,i),X_0^u(T;t,0,i)\rangle\\
    &&\qq+\int_t^T\[\Blan Q(s,\a(s))X_0^u(s;t,0,i),X_0^u(s;t,0,i)\Bran  \\
    &&\qq\qq\qq+2\Blan S(s,\a(s))X_0^u(s;t,0,i),u(s)\Bran+\Blan R(s,\a(s))u(s),u(s)\Bran\]ds\bigg\}\\
    &\ges&\dbE\int_t^T\[\Blan Q(s,\a(s))X_0^u(s;t,0,i),X_0^u(s;t,0,i)\Bran\\
    &&\qq\qq+2\Blan S(s,\a(s))X_0^u(s;t,0,i),u(s)\Bran+\Blan R(s,\a(s))u(s),u(s)\Bran\]ds\\
    &=&\dbE\int_t^T\[\Blan\big[Q(s,\a(s))\1n-\1nS(s,\a(s))^\top R(s,\a(s))^{-1}S(s,\a(s))\big]X_0^u(s;t,0,i),X_0^u(s;t,0,i)\Bran\1n\\
    &&\qq\qq+\Blan R(s,\a(s))\big[u(s)\1n+\1nR(s,\a(s))^{-1}S(s,\a(s))X_0^u(s;t,0,i)\big],\1n\\
    &&\qq\qq\qq\qq\qq\qq~u(s)+\1nR(s,\a(s))^{-1}S(s,\a(s))X_0^u(s;t,0,i)\Bran\]ds\\
    &\ges&\d\dbE\int_t^T\big|u(s)+R(s,\a(s))^{-1}S(s,\a(s))X_0^u(s;t,0,i)\big|^2
           ds\\
    &\ges&\d\g\dbE\int_t^T|u(s)|^2 ds,\q \forall
           u(\cd)\in\cU[t,T],
  \end{eqnarray*}
  for some $\g>0$. This completes the proof.
\end{proof}

Next, we shall show that the uniform convexity of $J^0(t,x,i;u(\cd)$ implies the open-loop solvability of Problem (M-SLQ).

\begin{thm}
  \label{sec:valueuniformconvex} \sl Let {\rm(H1)--(H2)} hold. Suppose the map
  $u(\cd)\mapsto J^0(t,0,i;u(\cd))$ is uniformly convex. Then Problem
  {\rm(M-SLQ)} is uniquely open-loop solvable, and there exists a
  constant $\g\in\dbR$ such that
  \bel{uni-convex-prop0}V^0(t,x,i)\ges\g|x|^2,\qq\forall
  (t,x)\in[0,T]\times\dbR^n.\ee
  Note that in the above, the constant $\g$ does not have to be
  nonnegative.
\end{thm}

\begin{proof}
  First of all, by the uniform convexity of
  $u(\cd)\mapsto J^0(t,0,i;u(\cd))$, we may assume that
  \begin{equation*}
    J^0(t,0,i;u(\cd))\ges\l\,\dbE\1n\int_t^T|u(s)|^2ds,\qq\forall
    u(\cd)\in\cU[0,T],
  \end{equation*}
  for some $\l>0$.
  % Now, for any $t\in[0,T)$, and any
  % $u(\cd)\in\cU[t,T]$, we define the {\it zero-extension} of $u(\cd)$
  % as follows:
  % %
  % \bel{ext}[\,0I_{[0,t)}\oplus u(\cd)](s)=\left\{\2n\ba{ll}0,\qq\ s\in[0,t),\\
  %   %
  %   \ns\ds u(s),\q s\in[t,T].\ea\right.\ee
  % %
  % Then $v(\cd)\equiv0I_{[0,t)}\oplus u(\cd)\in\cU[0,T]$, and due to
  % the initial state being 0, the solution $X(s)$ of
  % %
  % $$\left\{\2n\ba{ll}
  %   %
  %   \ns\ds dX(s)=\big[A(s,\a(s))X(s)+B(s,\a(s))v(s)\big]ds+\big[C(s,\a(s))X(s)+D(s,\a(s))v(s)\big]dW(s),\qq s\in[0,T], \\
  %   %
  %   \ns\ds X(0)=0,\ea\right.$$
  % %
  % satisfies
  % %
  % $$X(s)=0,\qq s\in[0,t].$$
  % %
  % Hence,
  % %
  % \begin{eqnarray}
  %   \label{4.4-RS}J^0(0,0,i;u(\cd))&=&\sum_{k=1}^DP_{ik}(t)J^0(t,0,k;0I_{[0,t)}\oplus u(\cd))\\
  %   &\ges&
  %   \l\,\dbE\1n\int_0^T\big|[0I_{[0,t)}\oplus
  %   u(\cd)](s)\big|^2ds=\l\,\dbE\1n\int_t^T|u(s)|^2ds.
  % \end{eqnarray}
  % %
  % \todo{How to obtain the uniformly convex of $J^0(t,x,i;u(\cd))$ from the uniformly convex of $J^0(t,0,i;u(\cd))$}
  Thus, $u(\cd)\mapsto J^0(t,x,i;u(\cd))$ is uniformly convex \todo{How to prove the uniform convexity of $J^0(t,x,i;u(\cd))$ from the uniform convexity of $J^0(t,0,i;u(\cd))$}for any
  given $(t,x)\in[0,T)\times\dbR^n$. By Corollary \ref{sec:frechetdifferential}, we have
  \bel{uni-convex-prop1}\ba{ll}
  \ns\ds J(t,x,i;u(\cd))=J(t,x,i;0)+J^0(t,0,i;u(\cd))+\dbE\int_t^T\lan\cD J(t,x,i;0)(s),u(s)\ran ds\\
  \ns\ds\ges J(t,x,i;0)+J^0(t,0,i;u(\cd))-{\l\over2}\dbE\int_t^T|u(s)|^2ds-{1\over2\l}\dbE\int_t^T|\cD J(t,x,i;0)(s)|^2ds\\
  \ns\ds\ges{\l\over2}\dbE\int_t^T|u(s)|^2ds+J(t,x,i;0)-{1\over2\l}\dbE\int_t^T|\cD
  J(t,x,i;0)(s)|^2ds,\qq\forall u(\cd)\in\cU[t,T].\ea\ee
  Consequently, % \todo{I did not find references for this step}
  by a standard argument involving minimizing sequence
  and locally weak compactness of Hilbert spaces, we see that for any
  given initial pair $(t,x,i)\in [0,T)\times\dbR^n\times\cS$, Problem (M-SLQ)
  admits a unique open-loop optimal control. Moreover, when $b(\cd), \si(\cd), g, q(\cd), \rho(\cd)=0$,
  (\ref{uni-convex-prop1}) implies that
  \bel{uni-convex-prop2}V^0(t,x,i)\ges J^0(t,x,i;0)-{1\over2\l}\dbE\int_t^T|\cD J^0(t,x,i;0)(s)|^2ds.\ee
  Note that the functions on the right-hand side of (\ref{uni-convex-prop2}) are
  quadratic in $x$ and continuous in $t$. (\ref{uni-convex-prop0})
  follows immediately.
\end{proof}

\section{Closed-loop Solvabilities}
In this section, we shall establish the equivalence between the closed-loop solvability and the existence of a regular solution to the Riccati equation. In the following, we first introduce some notation and the Riccati equation. % then present some lemmas which will be used
% frequently in sequel.
Let
\begin{equation}
  \label{eq:hatsr}
  \begin{aligned}
    \hat S(s,i)&:=B(s,i)^\top P(s,i)+ D(s,i)^\top P(s,i)C(s,i)+S(s,i),\\
    \hat R(s,i)&:=R(s,i)+D(s,i)^\top\1n P(s,i)D(s,i).
  \end{aligned}
\end{equation}
The Riccati equation associated with Problem (M-SLQ) is
\begin{equation}
  \label{Riccati}
  \left\{
    \begin{aligned}
      \dot P(s,i)&+P(s,i)A(s,i)+A(s,i)^\top P(s,i)+C(s,i)^\top P(s,i)C(s,i)\\
      &-\hat{S}(s,i)^\top\hat{R}(s,i)^\dag\hat{S}(s,i)+Q(s,i)+\sum_{k=1}^D\l_{ik}(s)P(s,k)=0,\q \ae~s\in[0,T],\\
      P(T,i)&=G(T,i).
    \end{aligned}
  \right.
\end{equation}

\begin{defn}
  A solution $P(\cd,\cd)\in C([0,T]\times \cS;\dbS^n)$ of (\ref{Riccati}) is said to
  be {\it regular} if
  \begin{equation}
    \label{eq:regular}
    \begin{aligned}
      \cR\big(\hat{S}(s,i)\big)&\subseteq\cR\big(\hat{R}(s,i)\big),\q
      \ae~s\in[0,T],\\
      \hat{R}(\cd,\cd)^\dag\hat{S}(\cd,\cd)&\in L^2(0,T;\dbR^{m\times n}),\\
      \hat{R}(s,i)&\ges0,\qq\ae~s\in[0,T].
    \end{aligned}
  \end{equation}
  % \begin{align}
      %       \label{regular-1}\cR\big(\hat{S}(s,i)\big)&\subseteq\cR\big(\hat{R}(s,i)\big),\q
                                                          %                                                           \ae~s\in[0,T],\\
      %       \label{regular-2}\hat{R}(\cd,\cd)^\dag\hat{S}(\cd,\cd)&\in L^2(0,T;\dbR^{m\times n}),\\
      %       \label{regular-3}\hat{R}(s,i)&\ges0,\qq\ae~s\in[0,T].
                                             %     \end{align}
  A solution $P(\cd,\cd)$ of (\ref{Riccati}) is said to be {\it strongly
    regular} if
  \begin{eqnarray}
    \label{strong-regular}\hat{R}(s,i)\ges \l  I,\qq\ae~s\in[0,T],
  \end{eqnarray}
  for some $\l>0$. The Riccati equation (\ref{Riccati}) is said to be
  ({\it strongly}) {\it regularly solvable}, if it admits a (strongly)
  regular solution. \end{defn}
Clearly, condition (\ref{strong-regular}) implies
(\ref{eq:regular}). Thus, a strongly regular
solution $P(\cd)$ must be regular. Moreover, %  it was shown in
% \citet{sun2014linear} that
if a regular solution of (\ref{Riccati})
exists, it must be unique.

% \bel{Riccati}\left\{\2n\ba{ll}
%   %
%   \ns\ds\dot P(s,i)+P(s,i)A(s,i)+A(s,i)^\top P(s,i)+C(s,i)^\top P(s,i)C(s,i)+Q(s,i)\\
%   %
%   \ns\ds\q-\big[P(s,i)B(s,i)+C(s,i)^\top P(s,i)D(s,i)+S(s,i)^\top\big]\big[R(s,i)+D(s,i)^\top P(s,i)D(s,i)\big]^\dag\\
%   %
%   \ns\ds\qq\cd\big[B(s,i)^\top P(s,i)+D(s,i)^\top P(s,i)C(s,i)+S(s,i)\big]+\sum_{k=1}^D\l_{ik}(s)P(s,k)=0,\q \ae~s\in[0,T],\\
%   %
%   \ns\ds P(T)=G.\ea\right.\ee
% %

\ms

\begin{thm} \sl
  \label{sec:closedloop-regusolu}  Let {\rm(H1)--(H2)} hold. Problem {\rm(M-SLQ)} is
  closed-loop solvable on $[0,T]$ if and only if the
  Riccati equation {\rm(\ref{Riccati})} admits a regular solution
  $P(\cd,\cd)\in C([0,T]\times\cS;\dbS^n)$ and the solution $(\eta(\cd),\z(\cd), \xi_1(\cd),\cds,\xi_D(\cd))$ of the
  following BSDE:
  \begin{eqnarray}
    \label{eta-zeta-xi}\left\{\2n\ba{ll}
    d\eta(s)=-\Big\{\big[A(s,\a(s))^\top\2n-\hat S(s,\a(s))^\top \hat R(s,\a(s))^\dag B(s,\a(s))^\top\big]\eta(s)\\
    \qq\qq\q+\big[C(s,\a(s))^\top\2n-\hat S(s,\a(s))^\top \hat R(s,\a(s))^\dag D(s,\a(s))^\top\big]\z(s)\\
    \ns\ds\qq\qq\q+\big[C(s,\a(s))^\top\2n-\hat S(s,\a(s))^\top \hat R(s,\a(s))^\dag D(s,\a(s))^\top\big]P(s,\a(s))\si(s,\a(s))\\
    \ns\ds\qq\qq\q-\hat S(s,\a(s))^\top \hat R(s,\a(s))^\dag \rho(s,\a(s))+P(s,\a(s))b(s,\a(s))+q(s,\a(s))\Big\}ds\\
    \ns\ds\qq\qq\q+\z(s) dW(s)+\sum_{k=1}^D\xi_k(s)d\wt{N}_k(s),\q s\in[0,T],\\
    \ns\ds\eta(T)=g(T,i),\ea\right.
  \end{eqnarray}
  satisfies
  \begin{eqnarray}
    \label{eta-zeta-regularity}\left\{\ba{ll}
    \hat \rho(s,i)\in\cR(\hat R(s,i)), \qq \ae~\as\\
    \ns\ds \hat R(s,i)^\dag\hat \rho(s,i)\in L_\dbF^2(0,T;\dbR^m),\ea\right.
  \end{eqnarray}
  with
  \begin{align}
    \label{eq:hatrrho}
    \hat \rho(s,i)&=B(s,i)^\top\eta(s)+D(s,i)^\top\z(s)+D(s,i)^\top P(s,i)\si(s,i)+\rho(s,i).
  \end{align}
  In this case, Problem {\rm(M-SLQ)} is
  closed-loop solvable on any $[t,T]$, and the closed-loop optimal
  strategy $(\Th^*(\cd),v^*(\cd))$ admits the following
  representation:
  \begin{align}
    \label{Th-v-rep}\left\{\2n\ba{ll}
    \Th^*(s)=-\hat R(s,\a(s))^\dag\hat S(s,\a(s))
               +\big[I-\hat R(s,\a(s))^\dag\hat R(s,\a(s))\big]\Pi,\\
    \ns\ds v^*(s)=-\hat R(s,\a(s))^\dag\hat\rho(s,\a(s))+\big[I-\hat R(s,\a(s))^\dag\hat R(s,\a(s))\big]\n(s),\ea\right.
  \end{align}
  for some $\Pi(\cd)\in L^2(t,T;\dbR^{m\times n})$ and $\n(\cd)\in
  L_\dbF^2(t,T;\dbR^m)$, and the value function is given by
  \begin{align}
    \label{Value}
    V(t,x,i)=\dbE\bigg\{&\langle P(t,i)x,x\rangle+2\langle\eta(t),x\rangle+\int_t^T\[\hat P(s,\a(s))-\lan\hat R(s,\a(s))^\dag\hat\rho(s,\a(s)),\hat\rho(s,\a(s)) \ran\]ds\bigg\},
  \end{align}
  where
  \begin{eqnarray*}
    \hat P(s,i):=\langle
    P(s,i)\si(s,i)+2\z(s),\si(s,i)\rangle+2\langle\eta(s),b(s,i)\rangle.
  \end{eqnarray*}

\end{thm}
\begin{proof}
  {\it Necessity.}
  Let $(\Th^*(\cd),v^*(\cd))$ be a closed-loop optimal strategy of Problem (M-SLQ) over $[t,T]$ and set
  \begin{align*}
    (X^*(\cd),Y^*(\cd),Z^*(\cd),\G^*(\cd)):=(X^{\Th^*,v^*}(\cd\,;t,x,i),Y^{\Th^*,v^*}(\cd\,;t,x,i),Z^{\Th^*,v^*}(\cd\,;t,x,i),\G^{\Th^*,v^*}(\cd\,;t,x,i)).
  \end{align*}
  % Then for any $x\in \dbR^n$, the following FBSDE admits an adapted solution $(X^*, Y^*, Z^*, \G^*)$:
  % $$ \left\{ \begin{aligned}
  %       %
  %     dX^*(s)=&\Big[(A(s,\a(s))+B(s,\a(s))\Th^*(s,\a(s))) X^*(s)+B(s,\a(s))v^*(s)+b(s,\a(s))\Big]ds\\
  %     &+\Big[(C(s,\a(s))+D(s,\a(s))\Th^*(s,\a(s)))X^*(s)+D(s,\a(s))v^*(s)+\si(s,\a(s))\Big]dW(s), \\
  %     dY^*(s)=&-\Big[A(s,\a(s))Y^*(s)+C(s,\a(s))^\top Z^*(s)+\big[Q(s,\a(s))+S(s,\a(s))\Th^*(s,\a(s))\big]X^*(s)\\
  %     &\qq+S(s,\a(s))^\top v^*(s)+q(s,\a(s))\Big]ds+Z^*(s)dW(s)+\sum_{k=1}^D\G_k^*(s)d\widetilde{N}_k(s), \q s\in[t,T], \\ 
  %     X^*(t)=& x, \qq Y^*(T)=G(T,\a(T))X^*(T)+g(T,\a(T)).
  %   \end{aligned}\right.
  % $$
  Then the following stationary condition hold:
  \begin{equation}
    \label{eq:statcondorig}
    \begin{aligned}
      B(s,\a(s))^\top Y^*(s)&+D(s,\a(s))^\top Z^*(s)+\big[S(s,\a(s))+R(s,\a(s))\Th^*(s)\big]X^*(s)\\
      &+R(s,\a(s))v^*(s)+\rho(s,\a(s))=0\quad \mbox{a.e. a.s.}
    \end{aligned}
  \end{equation}

  Since the above admits a solution for each $x\in\dbR^n$, and $(\Theta^*(\cd,\cd),v^*(\cd))$ is independent of $x$, by subtracting soulutions corresponding to $x$ and $0$, the later from the former, we see that for any $x\in \dbR^n$, as long as $(X(\cd),Y(\cd),Z(\cd),\G(\cd))$ is the adapted solution to the FBSDE
  $$ \left\{ \begin{aligned}
      dX(s)=&\big[A(s,\a(s))+B(s,\a(s))\Th^*(s)\big] X(s)ds+\big[C(s,\a(s))+D(s,\a(s))\Th^*(s)\big]X(s)dW(s), \\
      dY(s)=&-\Big[A(s,\a(s))^\top Y(s)+C(s,\a(s))^\top Z(s)+\big[Q(s,\a(s))+S(s,\a(s))^\top\Th^*(s)\big]X(s)\Big]ds\\
      &+Z(s)dW(s)+\sum_{k=1}^D\G_k(s)d\widetilde{N}_k(s), \q s\in[t,T], \\
      X(t)=& x, \qq Y(T)=G(T,\a(T))X(T),
    \end{aligned}\right.
  $$
  one must have the following stationary condition:
  \begin{equation}
    \label{eq:Newstationary}
    \begin{aligned}
      B(s,\a(s))^\top Y(s;t,x,i)&+D(s,\a(s))^\top Z(s;t,x,i)\\
      &+\big[S(s,\a(s))+R(s,\a(s))\Th^*(s)\big]X(s;t,x,i)=0\quad \mbox{a.e. a.s.},
    \end{aligned}
  \end{equation}
  where
  \begin{align*}
    &(X(\cd\,;t,x,i),Y(\cd\,;t,x,i),Z(\cd\,;t,x,i),\G(\cd\,;t,x,i))\\
    &:=(X_0^{\Th^*,0}(\cd\,;t,x,i),Y_0^{\Th^*,0}(\cd\,;t,x,i),Z_0^{\Th^*,0}(\cd\,;t,x,i),\G_0^{\Th^*,0}(\cd\,;t,x,i)).
  \end{align*}
  Let $e_i$ denote the unit vector of $\dbR^n$ whose $i$-th  component is one. Define, for $t\leq s\leq T$,
  \begin{align*}
    X(s;t,i):=&(X(s;t,e_1,i),\cdots,X(s;t,e_n,i))\\
    Y(s;t,i):=&(Y(s;t,e_1,i),\cdots,Y(s;t,e_n,i))\\
    Z(s;t,i):=&(Z(s;t,e_1,i),\cdots,Z(s;t,e_n,i))\\
    \G_{k}(s;t,i):=&(\G_{k}(s;t,e_1,i),\cdots,\G_{k}(s;t,e_n,i)).
  \end{align*}
  It is easy to verify that
  \begin{equation}
    \label{eq:RelationXYZ}
    \begin{aligned}
      & X(s;t,x,i)=X(s;t,i)x, \q Y(s;t,x,i)=Y(s;t,i)x,\\
      &Z(s;t,x,i)=Z(s;t,i)x, \q \G_{k}(s;t,x,i)=\G_{k}(s;t,i)x.
    \end{aligned}
  \end{equation}
  In particular, if we set $P(t,i):=Y(t;t,i)$, then
  \begin{align*}
    Y(t;t,x,i)=Y(t;t,i)x=P(t,i)x.
  \end{align*}
  Therefore,
  \begin{align*}
    Y(s;t,i)x&=Y(s;t,x,i)=Y(s;s,X(s;t,x,i),\a(s))=Y(s;s,\a(s))X(s;t,x,i)\\
             &=P(s,\a(s))X(s;t,i)x,\qq \mbox{for any } x\in\dbR^n,
  \end{align*}
  which leads to
  \begin{equation}
    \label{eq:YP}
    Y(s;t,i)=P(s,\a(s))X(s;t,i).
  \end{equation}
  Applying the It{\^o}'s formula to $P(s,\a(s))X(s;t,i)$ yields
  \begin{equation}
    \label{eq:PX}
    \begin{aligned}
      d[P(s,\a(s))X(s;t,i)]=&\bigg[\dot P(s,\a(s))+P(s,\a(s))\big[A(s,\a(s))+B(s,\a(s))\Th^*(s)\big]\\
      &\q+\sum_{k=1}^D\l_{\a(s-)k}(s)\big[P(s,k)-P(s,\a(s-))\big]\bigg]X(s;t,i)ds\\
      &+P(s,\a(s))\big[C(s,\a(s))+D(s,\a(s))\Th^*(s)\big]X(s;t,i)dW(s)\\
      &+\sum_{k=1}^D\big[P(s,k)-P(s,\a(s-))\big]X(s;t,i)d\wt{N}_k(s)
    \end{aligned}
  \end{equation}
  Observing that $Y(s;t,i)$ satisfied the following SDE
  \begin{equation}
    \label{eq:Y}
    \left\{
      \begin{aligned}
        dY(s;t,i)=&-\Big[A(s,\a(s))^\top Y(s;t,i)+C(s,\a(s))^\top Z(s;t,i)\\
        &\q+\big[Q(s,\a(s))+S(s,\a(s))\Th^*(s)\big]X(s;t,i)\Big]ds\\
        &+Z(s;t,i)dW(s)+\sum_{k=1}^D\G_k(s;t,i)d\widetilde{N}_k(s), \q s\in[0,T], \\
        Y(T;0,i)=&G(T,\a(T))X(T;0,i).
      \end{aligned}\right.
  \end{equation}
  Comparing the coefficients of \eqref{eq:PX} and \eqref{eq:Y}, we must have
  \begin{equation}
    \label{eq:ZGP}
    \begin{aligned}
      Z(s;t,i)=&P(s,\a(s))\big[C(s,\a(s))+D(s,\a(s))\Th^*(s)\big]X(s;t,i),\\
      \G_k(s;t,i)=&\big[P(s,k)-P(s,\a(s-))\big]X(s;t,i),
    \end{aligned}
  \end{equation}
  and
  \begin{equation}
    \label{eq:7}
    \begin{aligned}
      &\bigg\{\dot P(s,\a(s))+A(s,\a(s))^\top P(s,\a(s))+P(s,\a(s))A(s,\a(s))+C(s,\a(s))^\top P(s,\a(s))C(s,\a(s))\\
      &\q+\bigg[P(s,\a(s))B(s,\a(s))+C(s,\a(s))^\top P(s,\a(s))D(s,\a(s))+S(s,\a(s))^\top\bigg]\Th^*(s)+Q(s,\a(s))\\
      &\q+\sum_{k=1}^D\l_{\a(s-)k}(s)\big[P(s,k)-P(s,\a(s-))\big]\bigg\}X(s;t,i)=0,
    \end{aligned}
  \end{equation}
  where the last equation leads to
  % \begin{equation}
  %   \label{eq:Pdiff}
  %   \begin{aligned}
  %     &P'(s,i)+A(s,i)^\top P(s,i)+P(s,i)A(s,i)+C(s,i)^\top P(s,i)C(s,i)\\
  %     &\q+\bigg[P(s,i)B(s,i)+C(s,i)^\top P(s,i)D(s,i)+S(s,i)^\top\bigg]\Th^*(s,i)+Q(s,i)\\
  %     &\q+\sum_{k=1}^D\l_{ik}(s)\big[P(s,k)-P(s,i)\big]=0.
  %   \end{aligned}
  % \end{equation}
\begin{equation}\label{eq:Pdiff}
    \begin{aligned}
      &\dot P(s,\a(s))+A(s,\a(s))^\top P(s,\a(s))+P(s,\a(s))A(s,\a(s))+C(s,\a(s))^\top P(s,\a(s))C(s,\a(s))\\
      &\q+\bigg[P(s,\a(s))B(s,\a(s))+C(s,\a(s))^\top P(s,\a(s))D(s,\a(s))+S(s,\a(s))^\top\bigg]\Th^*(s)+Q(s,\a(s))\\
      &\q+\sum_{k=1}^D\l_{\a(s-)k}(s)\big[P(s,k)-P(s,\a(s-))\big]=0.
    \end{aligned}
  \end{equation}
  From \eqref{eq:RelationXYZ}, \eqref{eq:YP} and \eqref{eq:ZGP}, and the definition  of $\hat{S}(\cd,\cd)$ and $\hat{R}(\cd,\cd)$ in \eqref{eq:hatsr},  the stationary condition \eqref{eq:Newstationary} can be rewritten as
  \begin{equation*}
    \big[\hat{S}(s,\a(s))+\hat{R}(s,\a(s))\Th^*(s)\big]X(s;t,i)=0\quad \mbox{a.e. a.s.},
  \end{equation*}
  which yields
  \begin{equation}
    \label{eq:statcond2}
    \hat{S}(s,\a(s)) +\hat{R}(s,\a(s))\Th^*(s)=0, \q i\in \dbS, \quad \mbox{a.e.}.
  \end{equation}
  This implies
  \begin{align*}
    \cR\big(\hat{S}(s,i)\subseteq\cR\big(\hat{R}(s,i)\big),\q
    \ae~s\in[0,T].
  \end{align*}
  % \begin{equation*}
  %   \begin{aligned}
  %     \bigg[&B(s,\a(s))^\top P(s,\a(s)) +D(s,\a(s))^\top P(s,\a(s))C(s,\a(s))+S(s,\a(s))\\
  %     &\qq+\big[R(s,\a(s))+D(s,\a(s))^\top P(s,\a(s))D(s,\a(s))\big]\Th^*(s,\a(s))\bigg]X(s;t,i)=0\quad \mbox{a.e. a.s.},
  %   \end{aligned}
  % \end{equation*}
  % which yields
  % \begin{equation}
  %   \label{eq:statcond2}
  %   \begin{aligned}
  %     &B(s,\a(s))^\top P(s,\a(s)) +D(s,\a(s))^\top P(s,\a(s))C(s,\a(s))+S(s,\a(s))\\
  %     &\qq+\big[R(s,\a(s))+D(s,\a(s))^\top P(s,\a(s))D(s,\a(s))\big]\Th^*(s,\a(s))=0\quad \mbox{a.e. a.s.}.
  %   \end{aligned}
  % \end{equation}
  % This implies
  % \begin{align*}
  %   &\cR\big(B(s,i)^\top P(s,i)+D(s,i)^\top P(s,i)C(s,i)+S(s,i)\big)\\
  %   &\subseteq\cR\big(R(s,i)+D(s,i)^\top P(s,i)D(s,i)\big),\q
  %   \ae~s\in[0,T],\nonumber
  % \end{align*}
Using \eqref{eq:statcond2}, we can rewrite \eqref{eq:Pdiff} as
  \begin{equation}\label{eq:Pdiff2}
    \begin{aligned}
      \dot P(s,\a(s))&+\big[A(s,\a(s))+B(s,\a(s))\Th^*(s)\big]^\top P(s,\a(s))\\
      &+P(s,\a(s))\big[A(s,\a(s))+B(s,\a(s))\Th^*(s)\big]\\
      &+\big[C(s,\a(s))+D(s,\a(s))\Th^*(s)\big]^\top P(s,\a(s))\big[C(s,\a(s))+D(s,\a(s))\Th^*(s)\big]\\
      &+\Th^*(s)^\top R(s,\a(s))\Th^*(s)+S(s,\a(s))^\top \Th^*(s)+ \Th^*(s)^\top S(s,\a(s))\\
      &+Q(s,\a(s))+\sum_{k=1}^D\l_{\a(s-)k}(s)\big[P(s,k)-P(s,\a(s-))\big]=0.
    \end{aligned}
  \end{equation}
  Since $P(T,i)=G(T,i)\in \dbS^n$ and $Q(\cd,\cd), R(\cd,\cd)$ are symmetric, we must have $P(\cd,\cd)\in C([t,T]\times S; \dbS^n) $ due to the uniqueness of the solution of \eqref{eq:Pdiff2}. Let $\hat{R}(\cd,\cd)^\dag$ be the pseudo inverse of $\hat{R}(\cd,\cd)$, then the solution of  \eqref{eq:statcond2} admits the following representation
  \begin{equation}
    \label{eq:3}
    \Th^*(s)=-\hat{R}(s,\a(s))^\dag \hat{S}(s,\a(s))+\big(I-\hat{R}(s,\a(s))^\dag\hat{R}(s,\a(s))\big)\Pi(s,\a(s)),
  \end{equation}
  for some $\Pi(\cd,\cd)\in L^2(t,T;\dbR^{m\times n})$. Noting that
  \begin{equation}
    \label{eq:stheta}
    \begin{aligned}
      \hat{S}(s,\a(s))^\top\Th^*(s)&=-\Th^*(s)\hat{R}(s,\a(s))\Th^*(s)\\
      &=-\Th^*(s)\hat{R}(s,\a(s))\big[-\hat{R}(s,\a(s))^\dag \hat{S}(s,\a(s))+\big(I-\hat{R}(s,\a(s))^\dag\hat{R}(s,\a(s))\big)\Pi(s,\a(s))\big]\\
      &=-\hat{S}(s,\a(s))^\top\hat{R}(s,\a(s))^\dag\hat{S}(s,\a(s))
    \end{aligned}
  \end{equation}
  Observing  $\sum_{k=1}^D\l_{ik}(s)=0$ and substituting the above equation into \eqref{eq:Pdiff}, we obtain
\begin{equation}\label{eq:Pdiff3}
    \begin{aligned}
      \dot P(s,\a(s))&+A(s,\a(s))^\top P(s,\a(s))+P(s,\a(s))A(s,\a(s))\\
      &+C(s,\a(s))^\top P(s,\a(s))C(s,\a(s))-\hat{S}(s,\a(s))^\top\hat{R}(s,\a(s))^\dag\hat{S}(s,\a(s))\\
      &+Q(s,\a(s))+\sum_{k=1}^D\l_{\a(s-)k}(s)P(s,k)=0,
    \end{aligned}
  \end{equation}
  which is equivalent to the Riccati equation \eqref{Riccati}.

  In the next, we try to determine $v^*(\cd)$. Let
  \begin{equation*}
    \left\{
      \begin{aligned}
        \eta(s)=Y^*(s)&-P(s,\a(s))X^*(s)\\
        \z(s)=Z^*(s)&-P(s,\a(s))[C(s,\a(s))+D(s,\a(s))\Th^*(s)]X^*(s) \q s\in[t,T]\\
        &-P(s,\a(s))D(s,\a(s))v^*(s)-P(s,\a(s))\si(s,\a(s))\\
        \xi_k(s)=\G_k^*(s)&-\big[P(s,k)-P(s,\a(s-))\big]X^*(s).
      \end{aligned}
    \right.
  \end{equation*}
  Then
  \begin{eqnarray}\label{eq:etas}
      d\eta(s)&=&dY^*(s)-dP(s,\a(s))\cd X^*(s)-P(s,\a(s))dX^*(s)\nonumber\\
      &=&-\bigg[A(s,\a(s))^\top Y^*(s)+C(s,\a(s))^\top Z^*(s)+\big(Q(s,\a(s))+S(s,\a(s))^\top\Th^*(s)\big)X^*(s)\nonumber\\
               &&\qq+S(s,\a(s))^\top v^*(s)+q(s,\a(s))\bigg]ds+Z^*(s)dW(s)+\sum_{k=1}^D\G_k^*(s)d\widetilde{N}_k(s)\nonumber\\
               &&+\bigg\{\bigg[A(s,\a(s))^\top P(s,\a(s))+P(s,\a(s))A(s,\a(s))+C(s,\a(s))^\top P(s,\a(s))C(s,\a(s))\nonumber\\
               &&\qq-\hat{S}(s,\a(s))^\top\hat{R}(s,\a(s))^\dag\hat{S}(s,\a(s))+Q(s,\a(s))\bigg]X^*(s)\nonumber\\
               &&\qq-P(s,\a(s))\bigg[\bigg(A(s,\a(s))+B(s,\a(s))\Th^*(s)\bigg)X^*(s)\nonumber\\
               &&\qq+B(s,\a(s))v^*(s)+b(s,\a(s))\bigg]\bigg\}ds\nonumber\\
               &&-P(s,\a(s))\bigg[\bigg(C(s,\a(s))+D(s,\a(s))\Th^*(s)\bigg)X^*(s)+D(s,\a(s))v^*(s)\nonumber\\
               &&\qq\qq\qq\q+\si(s,\a(s))\bigg]dW(s)-\sum_{k=1}^D\big[P(s,k)-P(s,\a(s-))\big]X^*(s)d\wt{N}_k(s)\nonumber\\
      &=&-\bigg[A(s,\a(s))^\top \eta(s)+C(s,\a(s))^\top \z(s)+\hat{S}(s,\a(s))^\top\big[\Th^*(s)X^*(s)+v^*(s)\big]\nonumber\\
               &&\qq+C(s,\a(s))^\top P(s,\a(s))\si(s,\a(s))+P(s,\a(s))b(s,\a(s))+q(s,\a(s))\nonumber\\
               &&\qq+\hat{S}(s,\a(s))^\top\hat{R}(s,\a(s))^\dag\hat{S}(s,\a(s))X^*(s)\bigg]ds+\z(s)dW(s)+\sum_{k=1}^D\xi_k(s)\wt{N}_k(s)\nonumber\\
      &=&-\bigg[A(s,\a(s))^\top \eta(s)+C(s,\a(s))^\top \z(s)+\hat{S}(s,\a(s))^\top v^*(s)+C(s,\a(s))^\top P(s,\a(s))\si(s,\a(s))\\
              &&\qq+P(s,\a(s))b(s,\a(s))+q(s,\a(s))\bigg]ds+\z(s)dW(s)+\sum_{k=1}^D\xi_k(s)\wt{N}_k(s),  \nonumber               
    % \begin{aligned}
    %      \end{aligned}
  \end{eqnarray}
  where the last equality follows from the equation \eqref{eq:stheta}.

  According to \eqref{eq:statcondorig}, we have
  \begin{align*}
    0=&    B(s,\a(s))^\top Y^*(s)+D(s,\a(s))^\top Z^*(s)\\
      &+\big[S(s,\a(s))+R(s,\a(s))\Th^*(s)\big]X^*(s)+R(s,\a(s))v^*(s)+\rho(s,\a(s))\\
    =&B(s,\a(s))^\top \big[\eta(s)+P(s,\a(s))X^*(s)\big]\\
      &+D(s,\a(s))^\top \bigg\{\z(s)+P(s,\a(s))\big[C(s,\a(s))+D(s,\a(s))\Th^*(s)\big]X^*(s)\\
      &-P(s,\a(s))D(s,\a(s))v^*(s)-P(s,\a(s))\si(s,\a(s))\bigg\}\\
      &+\big[S(s,\a(s))+R(s,\a(s))\Th^*(s)\big]X^*(s)+R(s,\a(s))v^*(s)+\rho(s,\a(s))\\
    =&\big[\hat{S}(s,\a(s))+\hat{R}(s,\a(s))\Th^*(s)]X^*(s)+\hat{\rho}(s,\a(s))+\hat{R}(s,\a(s))v^*(s)\\
    =&\hat{\rho}(s,\a(s))+\hat{R}(s,\a(s))v^*(s),
       %        B(s,\a(s))^\top\eta(s)\\
      %                                                         &+D(s,\a(s))^\top\z(s)+D(s,\a(s))^\top P(s,\a(s))\si(s,\a(s))+\rho(s,\a(s))+\hat{R}(s,\a(s))v^*(s)\\
      %       =&B(s,\a(s))^\top\eta(s)+D(s,\a(s))^\top\z(s)+D(s,\a(s))^\top P(s,\a(s))\si(s,\a(s))+\rho(s,\a(s))+\hat{R}(s,\a(s))v^*(s)
  \end{align*}
  where $\h \rho(s,i)$ is defined by \eqref{eq:hatrrho}. Thus we have
  \begin{equation*}
    \hat{\rho}(s,i)\in \cR(\hat{R}(s,i)),
  \end{equation*}
  and
  \begin{equation*}
    v^*(s)=-\hat{R}(s,\a(s))^\dag\hat{\rho}(s,\a(s))+\big[I-\hat{R}(s,\a(s))^\dag\hat{R}(s,\a(s))]\n(s,\a(s)),
  \end{equation*}
  for some $\n(\cd,i)\in L_\dbF^2(t,T;\dbR^m)$. Consequently,
  \begin{equation*}
    \begin{aligned}
      \hat{S}(s,\a(s))^\top v^*(s)&=-\Th^*(s)^\top \hat{R}(s,\a(s))v^*(s)\\
      &=\Th^*(s)^\top \hat{R}(s,\a(s)) \hat{R}(s,\a(s))^\dag \hat{\rho}(s,\a(s))\\
      &=-\hat{S}(s,\a(s))^\top\hat{R}(s,\a(s))^\dag \hat{\rho}(s,\a(s)).
    \end{aligned}
  \end{equation*}
  Thus observing the definition of $\hat{\rho}(s,\a(s))$ and substituting the above equation into \eqref{eq:etas} yield the desired result of equation \eqref{eta-zeta-xi}.

  {\it Sufficiency.}  % To simply our proof, we introduce the following notation:
  % \begin{eqnarray*}
  %   \wt T_\a^0X(\cd)&:=&\bigg[\dot P(\cd,\a(\cd))+P(\cd,\a(\cd))A(\cd,\a(\cd))+A(\cd,\a(\cd))^\top P(\cd,\a(\cd))\\
  %   &&\q +C(\cd,\a(\cd))^\top P(\cd,\a(\cd))C(\cd,\a(\cd))+Q(\cd,\a(\cd))+\sum_{k=1}^D\l_{\a(\cd),k}(\cd)P(\cd,k)\bigg]X(\cd). \\
  %   \wt T_\a^1u(\cd) &:=&\big[R(\cd,\a(\cd))+D(\cd, \a(\cd))^\top P(\cd, \a(\cd))D(\cd, \a(\cd))\big]u(\cd)\\
  %   \wt T_\a^2X(\cd) &:=&\big[B(\cd,\a(\cd))^\top P(\cd,\a(\cd))+D(\cd,\a(\cd))^\top P(\cd,\a(\cd))C(\cd,\a(\cd))+S(\cd,\a(\cd))\big]X(\cd)
  % \end{eqnarray*}
  Applying It{\^o}'s formula to $s\mapsto \langle P(s,\a(s))X(s)+2\eta(s),X(s)\rangle$ yields
  \begin{equation}
    \label{eq:Jtxiu}
    \begin{aligned}
      &J(t,x,i;u(\cd))\\
      &=\dbE \bigg\{\lan P(t,i)x+2\eta(t),x\ran+\int_t^T\bigg[\lan P(s,\a(s))\si(s,\a(s))+2\z(s),\si(s,\a(s))\ran +2\lan\eta(s),b(s,\a(s))\ran\bigg]ds\\
      &\qq\q+\int_t^T\bigg[\Blan \hat Q(s,\a(s))X(s),X(s)\Bran+\Blan \hat R (s,\a(s))u(s)+2\hat S(s,\a(s)) X(s)+2\hat{\rho}(s,\a(s)),u(s)\Bran\\
      &\qq\qq\qq\q+2\Blan\hat{S}(s,\a(s))^\top \hat{R}(s,\a(s))^\dag \hat\rho(s,\a(s)),X(s)\Bran\bigg]ds\bigg\},
    \end{aligned}
  \end{equation}
  where
  \begin{equation}
    \label{eq:hatQ}
    \begin{aligned}
      \h Q(s,i)&:=\dot P(s,i)+P(s,i)A(s,i)+A(s,i)^\top P(s,i)\\
      &\qq+C(s,i)^\top P(s,i)C(s,i)+Q(s,i)+\sum_{k=1}^D\l_{ik}(s)P(s,k).
    \end{aligned}
  \end{equation}
  Let $\Th^*(\cd)$ and $v^*(\cd)$ be defined by \eqref{Th-v-rep}. It is easy to verify that
  \begin{align*}
    \hat S(s,\a(s))&=-\hat R(s,\a(s))\Th^*(s),\\
    \hat Q(s,\a(s))&=\Th^*(s)^\top \hat R(s,\a(s))\Th^*(s),\\
    \hat \rho(s,\a(s))&=-\hat R(s,\a(s))v^*(s),\\
    -\hat S(s,\a(s))^\top \hat R(s,\a(s))^\dag\hat \rho(s,\a(s))&=-\Th^*(s)^\top \hat R(s,\a(s))v^*(s).
  \end{align*}
Substituting these equation into \eqref{eq:Jtxiu} yields

\begin{align*}
  &J(t,x,i;u(\cd))\\
  &=\dbE \bigg\{\lan P(t,i)x+2\eta(t),x\ran +\int_t^T\bigg[\lan P(s,\a(s))\si(s,\a(s))+2\z(s),\si(s,\a(s))\ran +2\lan\eta(s),b(s,\a(s))\ran\bigg]ds  \\
  &\qq\q+\int_t^T\bigg[\Blan \Th^*(s)^\top \hat R(s,\a(s))\Th^*(s)X(s),X(s)\Bran\\
  &\qq\qq\qq\q+\Blan\hat R(s,\a(s))u(s)-2 \hat R (s,\a(s))\big[\Th^*(s)X(s)+v^*(s)\big],u(s)\Bran \\
  &\qq\qq\qq\q+2\Blan\Th^*(s)^\top \hat R(s,\a(s))v^*(s),X(s)\Bran \bigg]ds\bigg\}\\
  &=\dbE \bigg\{\lan P(t,i)x+2\eta(t),x\ran +\int_t^T\bigg[\lan P(s,\a(s))\si(s,\a(s))+2\z(s),\si(s,\a(s))\ran \\
  &\qq\q+2\lan\eta(s),b(s,\a(s))\ran-\lan\hat R(s,\a(s))v^*(s),v^*(s)\ran\bigg]ds\\
  &\qq\q+\int_t^T\Blan \hat R(s,\a(s))\big[u(s)-\Th^*(s)X(s)-v^*(s)\big],u(s)-\Th^*(s)X(s)-v^*(s)\Bran ds\Bigg\}\\
  &=J(t,x,i;\Th^*(\cd)X^*(\cd)+v^*(\cd))\\
  &\qq+\dbE\int_t^T\Blan \hat R(s,\a(s))\big[u(s)-\Th^*(s)X(s)-v^*(s)\big],u(s)-\Th^*(s)X(s)-v^*(s)\Bran ds.
\end{align*}
For any $v(\cd)\in \cU[t,T]$, let $u(\cd):=\Th^*(\cd)X(\cd)+v(\cd)$ with $X(\cd)$ being the solution to the state equation under the closed-loop strategy $(\Th^*(\cd),v(\cd))$. Then the above implies that
\begin{align*}
J(t,x,i;\Th^*(\cd)X(\cd)+v(\cd))=&J(t,x,i;\Th^*(\cd)X^*(\cd)+v^*(\cd))\\
                                 &+\dbE\int_t^T\lan \hat R(s,\a(s))\big[v(s)-v^*(s)], v(s)-v^*(s)\ran ds.
\end{align*}
Therefore, $(\Th^*(\cd),v^*(\cd))$ is a closed-loop optimal strategy if and only if
\begin{align*}
\dbE\int_t^T\lan \hat R(s,\a(s))\big[v(s)-v^*(s)], v(s)-v^*(s)\ran ds\ge 0, \q\forall v(\cd)\in\cU[t,T],
\end{align*}
or equivalently,
\begin{align*}
\hat R(s,\a(s))\ge 0, \q a.e. s\in[t,T].
\end{align*}
Finally, the representation of the value function follows from the identity
\begin{align*}
\lan \hat R(s,\a(s))v^*(s),v^*(s)\ran=\lan\hat R(s,\a(s))^\dag \hat \rho(s,\a(s)),\hat\rho(s,\a(s))\ran.
\end{align*}
\end{proof}

% \ms

% The following is concerned with the convexity of the cost
% functional, whose proof is straightforward, by making use of the
% representation (\ref{J-rep1}) of the cost functional.

% \ms

% \begin{coro}
%   \label{equivconvex}
%   \sl Let {\rm(H1)--(H2)} hold and let $(t,i)\in[0,T)\times\cS$
%   be given. Then the following are equivalent:

%   \ms

%   {\rm(i)} $u(\cd)\mapsto J(t,x,i;u(\cd))$ is convex, for some
%   $x\in\dbR^n$.

%   \ms

%   {\rm(ii)} $u(\cd)\mapsto J(t,x,i;u(\cd))$ is convex, for any
%   $x\in\dbR^n$.

%   \ms

%   {\rm(iii)} $u(\cd)\mapsto J^0(t,x,i;u(\cd))$ is convex, for some
%   $x\in\dbR^n$.

%   \ms

%   {\rm(iv)} $u(\cd)\mapsto J^0(t,x,i;u(\cd))$ is convex, for any
%   $x\in\dbR^n$.

%   \ms

%   {\rm(v)} $J^0(t,0,i;u(\cd))\ges0$, for all $u(\cd)\in\cU[t,T]$.

%   \ms

%   {\rm(vi)} $M_2(t,i)\ges0$.

% \end{coro}

% \rm

% \ms

\section{Uniform convexity of the cost functional and the strongly regular solution  of the Riccati equation}
% In the following, we first introduce some notations and then present some lemmas which will be used
% frequently in sequel. Let
% \begin{equation}
%   \label{eq:hatsr}
%   \begin{aligned}
%     \hat S(s,i)&:=B(s,i)^\top P(s,i)+ D(s,i)^\top P(s,i)C(s,i)+S(s,i),\\
%     \hat R(s,i)&:=R(s,i)+D(s,i)^\top\1n P(s,i)D(s,i),
%   \end{aligned}
% \end{equation}
We first present some properties for the solution to Lyapunov equation, which play a crucial role on establishing the equivalence between uniform convexity of the cost functional and the strongly regular solution of the Riccati equation.
\begin{lem}
  \label{sec:reprecostclosedloop}
  \rm Let {\rm(H1)--(H2)} hold and $\Th(\cd)\in
  L^2(0,T;\dbR^{m\times n})$ for $i\in \cS$. Let $P(\cd,i)\in C([0,T];\dbS^n), i\in\cS$ be the
  solution to the following Lyapunov equation:
  \begin{equation}
    \label{P-Th}
    \left\{\begin{aligned}
        \dot P(s,i)&+P(s,i)A(s,i)+A(s,i)^\top P(s,i)+C(s,i)^\top P(s,i)C(s,i)\\
        &+\hat S(s,i)^\top\Th(s)+\Th(s)^\top \hat S(s,i)+\Th(s)^\top \hat R(s,i)\Th(s)\\
        &+Q(s,i)+\sum_{k=1}^D\lambda_{ik}(s)P(s,k)=0,\qq\ae~s\in[0,T],\\
        P(T,i)&=G(T,i).
      \end{aligned}\right.
  \end{equation}
  Then for any $(t,x,i)\in[0,T)\times\dbR^n\times\cS$ and $u(\cd,\cd)\in\cU[t,T]$, we have
  \begin{equation*}
    \begin{aligned}
      &J^0(t,x,i;\Th(\cd)X_0^{\Th,u}(\cd\,;t,x,i)+u(\cd))=\langle P(t,i)x,x\rangle+\dbE\int_t^T\Big\{\lan T_\a^1u(s),u(s)\ran+2\lan T_\a^2X_0^{\Th,u}(s;t,x,i),u(s)\ran\Big\}ds.
    \end{aligned}
  \end{equation*}
  where $X_0^{\Th,u}(\cd\,;t,x,i)$ is the solution of \eqref{generalstate} and
  \begin{eqnarray*}
    T_\a^1u(\cd) &:=&\hat R(\cd,\a(\cd))u(\cd)\\
    T_\a^2X_0^{\Th,u}(\cd;t,x,i) &:=&\big[\hat S(\cd,\a(\cd))+\hat R(\cd,\a(\cd))\Th(\cd)\big]X_0^{\Th,u}(\cd\,;t,x,i).
  \end{eqnarray*}
\end{lem}
% \bf Lemma 2.3.
% \begin{eqnarray*}
    %     T_\a^1u(s) &:=&\big[R(s, \a(s))+D(s, \a(s))^\top P(s, \a(s))D(s, \a(s))\big]u(s)\\
    %     T_\a^2X^{x,u}(s) &:=&\big[B(s,\a(s))^\top P(s,\a(s))+D(s,\a(s))^\top P(s,\a(s))C(s,\a(s))+S(s,\a(s))\\
    %                      &&\q+(R(s,\a(s))+D(s,\a(s))^\top P(s,\a(s))D(s,\a(s)))\Th(s,\a(s))\big]X^{x,u}(s)
                              %   \end{eqnarray*}
                              %                               \ms
\begin{proof}
  \rm For any $(t,x)\in[0,T)\times\dbR^n$ and $u(\cd)\in\cU[t,T]$,
  let $X_0^{x,u}$ be the solution of \eqref{generalstate} % and
  % %
                              %                               $$\left\{\2n\ba{ll}
                              %   %
                              %                               \ns\ds dX(s)=\big[\big(A(s,\a(s))+B(s,\a(s))\Th(s,\a(s))\big)X(s)+B(s,\a(s))u(s)\big]ds\\
                              %   %
      %       \ns\ds \qq\qq\q+\big[\big(C(s,\a(s))+D(s,\a(s))\Th(s,\a(s))\big)X(s)+D(s,\a(s))u(s)\big]dW(s),\qq s\in[t,T],\\
      %   %
      %       \ns\ds X(t)=x,\ea\right.$$
      % %
  and set
  \begin{equation*}
    \begin{aligned}
      T_\a^0X_0^{\Th,u}(\cd\,;t,x,i):= \bigg[\dot P(\cd,\a(\cd))&+P(\cd,\a(\cd))A(\cd,\a(\cd))+A(\cd,\a(\cd))^\top P(\cd,\a(\cd))+C(\cd,\a(\cd))^\top P(\cd,\a(\cd))C(\cd,\a(\cd))\\
      &+\hat S(\cd,\a(\cd))^\top\Th(\cd)+\Th(\cd)^\top \hat S(\cd,\a(\cd))+\Th(\cd)^\top \hat R(\cd,\a(\cd))\Th(\cd)\\
      &+Q(\cd,\a(\cd))+\sum_{k=1}^D\lambda_{\a(\cd)k}(\cd)P(\cd,k)\bigg] X_0^{\Th,u}(\cd\,;t,x,i)
    \end{aligned}
  \end{equation*}
        %         \begin{eqnarray*}
        %         T_\a^0X_0^{\Th,u}(\cd\,;t,x,i)&:=&        \bigg[\dot P(\cd,\a(\cd))+P(\cd,\a(\cd))\big[A(\cd,\a(\cd))+B(\cd,\a(\cd))\Th(\cd,\a(\cd))\big]\\
    %                      &&\q+\big[A(\cd,\a(\cd))+B(\cd,\a(\cd))\Th(\cd,\a(\cd))\big]^\top P(\cd,\a(\cd))\\
    %                      &&\q+\big[C(\cd,\a(\cd))+D(\cd,\a(\cd))\Th(\cd,\a(\cd))\big]^\top P(\cd,\a(\cd))\big[C(\cd,\a(\cd))+D(\cd,\a(\cd))\Th(\cd,\a(\cd))\big]\\
    %                      &&\q +\Th(\cd,\a(\cd))^\top R(\cd,\a(\cd))\Th(\cd,\a(\cd))+\Th(\cd,\a(\cd))^\top S(\cd,\a(\cd))+S(\cd,\a(\cd))^\top\Th(\cd,\a(\cd))\\
    %                      &&\q+Q(\cd,\a(\cd))+\sum_{k=1}^D\l_{\a(\cd),k}(\cd)P(\cd,k)\bigg]X_0^{x,u,\Th}(\cd).
                              %   \end{eqnarray*}
  Applying It\^o's formula to $s\mapsto\langle P(s,\a(s))X(s),X(s)\rangle$, we have
  \begin{equation*}
    \begin{aligned}
      &J^0(t,x,i;\Th(\cd)X_0^{\Th,u}(\cd\,;t,x,i)+u(\cd))\\
      &=\dbE\Bigg\{\Blan G(T,\a(T))X_0^{\Th,u}(T;t,x,i),X_0^{\Th,u}(T;t,x,i)\Bran+\int_t^T\bigg[\Blan Q(s,\a(s))X_0^{\Th,u}(s;t,x,i), X_0^{\Th,u}(s;t,x,i)\Bran\\
      &\qq\qq+2\Blan\1nS(s,\a(s))X_0^{\Th,u}(s;t,x,i),u(s)\Bran +\Blan R(s,\a(s))u(s),u(s)\Bran\bigg] ds\Bigg\}\\
      &=\langle P(t,i)x,x\rangle+\dbE\int_t^T\Big\{\lan T_\a^0X_0^{\Th,u}(s;t,x,i),X_0^{\Th,u}(s;t,x,i)\ran+\lan T_\a^1u(s),u(s)\ran+2\lan T_\a^2X_0^{\Th,u}(s;t,x,i),u(s)\ran\Big\}ds\\
      &=\langle P(t,i)x,x\rangle+\dbE\int_t^T\Big\{\lan T_\a^1u(s),u(s)\ran+2\lan T_\a^2X_0^{\Th,u}(s;t,x,i),u(s)\ran\Big\}ds.
    \end{aligned}
  \end{equation*}
  This completes the proof.
\end{proof}

  \begin{prop}\rm
   \label{sec:closedlooplyapunov} Let {\rm(H1)--(H2)} and {\rm(\ref{J>l})}
    hold. Then for any $\Th(\cd)\in L^2(0,T;\dbR^{m\times n})$, the
    solution $P(\cd,\cd)\in C([0,T];\dbS^n)$ to the Lyapunov equation (\ref{P-Th})
    % %
    % \begin{eqnarray}
    %   \label{P-Th-l}\left\{\2n\ba{ll}
    %     %
    %   \ns\ds\dot P(s,i)+P(s,i)\big[A(s,i)+B(s,i)\Th(s)\big]+\big[A(s,i)+B(s,i)\Th(s)\big]^\top P(s,i)\\
    %   %
    %   \ns\ds\q+\big[C(s,i)+D(s,i)\Th(s)\big]^\top P(s,i)\big[C(s,i)+D(s,i)\Th(s)\big]+\Th(s)^\top R(s,i)\Th(s)\\
    %   %
    %   \ns\ds\q+S(s,i)^\top\Th(s)+\Th(s)^\top S(s,i)+Q(s,i)+\sum_{k=1}^D\lambda_{ik}(s)P(s,k)=0,\qq\ae~s\in[0,T],\\
    %   %
    %   \ns\ds P(T,i)=G(T,i).\ea\right.\hspace{-1.8cm}
    % \end{eqnarray}
    %  %
    satisfies
    \begin{eqnarray}
      \label{Convex-prop-1} \hat R(t,i)\ges\l I, \q\ae~t\in[0,T],\qq\hb{and}\qq P(t,i)\ges \g I,\q\forall t\in[0,T],
    \end{eqnarray}
    where $\g\in\dbR$ is the constant appears in
    {\rm(\ref{uni-convex-prop0})}.
  \end{prop}
\begin{proof}
  Let $\Th(\cd)\in L^2(0,T;\dbR^{m\times n})$ and let
  $P(\cd,\cd)$ be the solution to {\rm(\ref{P-Th})}. % For any
  % $u(\cd)\in\cU[t,T]$, let $X_0^{0,u,\Th}(\cd)$ be the solution of
  % %
  % \begin{align*}
  %   \left\{\2n\ba{l}
  %     %
  %   dX_0^{0,u,\Th}(s)=\big[(A(s,\a(s))+B(s,\a(s))\Th(s)) X_0^{0,u,\Th}(s)+B(s,\a(s))u(s)\big]ds\\
  %   \qq\qq\qq +\big[(C(s,\a(s))+D(s,\a(s))\Th(s))X_0^{0,u,\Th}(s)+D(s,\a(s))u(s)\big]dW(s),\qq s\in[t,T], \\
  %   %
  %    X_0^{0,u,\Th}(t)=0.\ea\right.
  % \end{align*}
  %  %
  By {\rm(\ref{J>l})} and Lemma \ref{sec:reprecostclosedloop}, we have
  $$\ba{ll}
  \ns\ds \l\dbE\int_t^T|\Th(s)X_0^{\Th,u}(s;t,0,i)+u(s)|^2ds\les J^0(t,0,i;\Th(\cd)X_0^{\Th,u}(\cd\,;t,0,i)+u(\cd))\\
  \ns\ds=\dbE\int_t^T\Big\{\lan \hat R(s,\a(s))u(s),u(s)\ran+2\lan [\hat S(s,\a(s))+\hat R(s,\a(s))\Th(s)]X_0^{\Th,u}(s;t,0,i),u(s)\ran\Big\}ds.\ea$$
  Hence, for any $u(\cd)\in\cU[t,T]$, the following holds:
  \begin{align}
    \label{P>LI}\ba{ll}
    \ns\ds\dbE\int_t^T\Big\{2\lan [\hat S(s,\a(s))+\big(\hat R(s,\a(s))-\l I\big)\Th(s)]X_0^{\Th,u}(s;t,0,i),u(s)\ran\\
    \ns\ds\qq\q~+\lan \big(\hat R(s,\a(s))-\l I\big)u(s),u(s)\ran\Big\}ds\ges \l\dbE\int_0^T|\Th(s)X_0^{\Th,u}(s;t,0,i)|^2ds\ges0.\ea
  \end{align}
  Let
  \begin{equation*}
    \F^\Th(\cd\,;t,i):=(X_0^{\Th,0}(\cd\,;t,e_1,i),\cds,X_0^{\Th,0}(\cd\,;t,e_n,i)).
  \end{equation*}
  Then it is easy to verify that $\F^\Th(\cd\,;t,i)$ is the solution to the following SDE for $\dbR^{n\times n}$-valued process:
  \begin{eqnarray}
    \label{FTh}\left\{\2n\ba{ll}
    \ns\ds d\F^\Th(s;t,i)=\big[A(s,\a(s))+B(s,\a(s))\Th(s)\big]\F^\Th(s;t,i)ds\\
    \qq\qq\q+\big[C(s,\a(s))+D(s,\a(s))\Th(s)\big]\F^\Th(s;t,i)dW(s),\qq s\ges0,\\
    \ns\ds\F^\Th(t;t,i)=I,\q \a(t)=i.\ea\right.
  \end{eqnarray}
Thus, $X_0^{\Th,u}(\cd\,;t,0,i)$ can be written as
\begin{align*}
  X_0^{\Th,u}(s;t,0,i)&=\F^\Th(s;t,i)\bigg\{\int_t^s\F^\Th(r;t,i)^{-1}\big[B(r,\a(r))-[C(r,\a(r))+D(r,\a(r))\Th(r)]D(r,\a(r))\big]u(r)dr\\
                      &\qq\qq\qq\q+\int_t^s\F^\Th(r;t,i)^{-1}D(r,\a(r))u(r)dW(r)\bigg\}.
\end{align*}
Now, fix any $u_0\in\dbR^m$, take $u(s)=u_0{\bf 1}_{[t,t+h]}(s)$, with $ 0\les t\les t+h\les T$.  Consequently, (\ref{P>LI}) becomes
\begin{align*}
 \ba{ll}
  \ns\ds\dbE\int_t^{t+h}\Big\{2\lan [\hat S(s,\a(s))+\big(\hat R(s,\a(s))-\l I\big)\Th(s)]\hat \F (s;t,i),u_0\ran+\lan \big(\hat R(s,\a(s))-\l I\big)u_0,u_0\ran\Big\}ds\ges0,\ea
\end{align*}
where
\begin{align*}
  \hat \F(s;t,i)&=\F^\Th(s;t,i)\bigg\{\int_t^s\F^\Th(r;t,i)^{-1}\big[B(r,\a(r))-[C(r,\a(r))+D(r,\a(r))\Th(r)]D(r,\a(r))\big]u_0dr\\
            &\qq\q\qq\qq+\int_t^s\F^\Th(r;t,i)^{-1}D(r,\a(r))u_0dW(r)\bigg\}.
\end{align*}
  Dividing both sides of the above by $h$ and letting $h\to 0$, we obtain
  $$\lan\big(\hat R(t,i)-\l I\big)u_0,u_0\ran\ges 0,\qq\ae~t\in[0,T],\q \forall u_0\in\dbR^m.$$
  The first inequality in (\ref{Convex-prop-1}) follows. To prove the second, for any $(t,x)\in[0,T)\times\dbR^n$ and $u(\cd)\in\cU[t,T]$ and
  % let $X_0^{x,u,\Th}(\cd)$ be the solution of
  % \begin{align*}
  %   \left\{\2n\ba{l}
  %     %
  %   dX_0^{x,u,\Th}(s)=\big[(A(s,\a(s))+B(s,\a(s))\Th(s)) X_0^{x,u,\Th}(s)+B(s,\a(s))u(s)\big]ds\\
  %   \qq\qq\qq +\big[(C(s,\a(s))+D(s,\a(s))\Th(s))X_0^{x,u,\Th}(s)+D(s,\a(s))u(s)\big]dW(s),\qq s\in[t,T], \\
  %   %
  %   X_0^{x,u,\Th}(t)=x.\ea\right.
  % \end{align*}
  %
  by Proposition \ref{sec:valueuniformconvex} and Lemma \ref{sec:reprecostclosedloop}, we have
  $$\ba{ll}
  \ns\ds \g|x|^2\les V^0(t,x,i)\les J^0(t,x,i;\Th(\cd)X_0^{\Th,u}(\cd\,;t,x,i)+u(\cd))\\
  \ns\ds\qq~\1n=\langle P(t,i)x,x\rangle\1n+\dbE\int_t^T\Big\{\lan \hat R(s,\a(s))u(s),u(s)\ran+2\lan [\hat S(s,\a(s))+\hat R(s,\a(s))\Th(s)]X_0^{\Th,u}(s;t,0,i),u(s)\ran\Big\}ds.\ea$$
  In particular, by taking $u(\cd)=0$ in the above, we obtain
  $$\langle P(t,i)x,x\rangle\ges\g|x|^2,\qq\forall (t,x,i)\in[0,T]\times\dbR^n\times \cS,$$
  and the second inequality therefore follows.
\end{proof}

Now we are in the position to prove the equivalence between the uniform convexity of the cost functional and the strongly regular solution of the Riccati equation.
  \begin{thm}
    \label{sec:unifconv-strongregusolu}    \sl Let {\rm(H1)--(H2)} hold. Then the following statements are equivalent:
    \ms

    {\rm(i)} The map $u(\cd)\mapsto J^0(t,0;u(\cd))$ is uniformly convex, i.e., there exists a
    $\l>0$ such that {\rm(\ref{J>l})} holds.

    \ms
    {\rm(ii)} The Riccati equation {\rm(\ref{Riccati})} admits a strongly regular solution $P(\cd,\cd)\in C([0,T]\times \cS;\dbS^n)$.
  \end{thm}

  \begin{proof} \rm (i) $\Ra$ (ii). Let $P_0(\cd,\cd)$ be the solution of
    \begin{align*}
      \left\{\ba{l}
      \dot P_0(s,i)+P_0(s,i)A(s,i)+A(s,i)^\top P_0(s,i)\\
      \qq\q\hspace{0.1cm}+C(s,i)^\top P_0(s,i)C(s,i)+Q(s,i)+\sum_{k=1}^D\l_{ik}(s)P_0(s,k)=0,\qq\ae~s\in[0,T],\\
      P_0(T,i)=G(T,i).\ea\right.
    \end{align*}
    Applying Proposition \ref{sec:closedlooplyapunov} with $\Th(\cd)=0$, we obtain that
    $$\hat R(s,i)\ges\l I,\q P_0(s,i)\ges\g I,\qq\ae~s\in[0,T].$$
    Next, inductively, for $n = 0,1,2, \cdots$, we set
    \begin{eqnarray}
      \label{Iteration-i}
      \left\{\ba{l}
      \Th_n(s,i)=-\hat R(s,i)^{-1}\big[B(s,i)^\top P_n(s,i)+D(s,i)^\top P_n(s,i)C(s,i)+S(s,i)\big],\\
      A_n(s,i)=A(s,i)+B(s,i)\Th_n(s,i),\\
      C_n(s,i)=C(s,i)+D(s,i)\Th_n(s,i),\ea\right.
    \end{eqnarray}
    and let $P_{n+1}$ be the solution of
    \begin{align*}
      \left\{\ba{l}
      \dot P_{n+1}(s,i)+P_{n+1}(s,i)A_n(s,i)+A_n(s,i)^\top P_{n+1}(s,i)\\
      \qq\q\hspace{0.5cm}+C_n(s,i)^\top P_{n+1}(s,i)C_n(s,i)+Q_n(s,i)+\sum_{k=1}^D\l_{ik}(s)P_{n+1}(s,k)=0,\qq\ae~s\in[0,T],\\
      P_{n+1}(T,i)=G(T,i).\ea\right.
    \end{align*}
    By Proposition \ref{sec:closedlooplyapunov}, we see that
    \begin{eqnarray}
      \label{R+Pi-lowerbound}\left\{\ba{l}
      R(s,i)+D(s,i)^\top P_{n+1}(s,i)D(s,i)\ges\l I,\\
      P_{n+1}(s,i)\ges\g I,\q\ae~s\in[0,T],\q n=0,1,2,\cdots.
      \ea\right.
    \end{eqnarray}
    We now claim that $\{P_n(s,i)\}_{n=1}^\i$ converges uniformly in
    $C([0,T];\dbS^n)$. To show this, let
    $$\D_n(s,i)\deq P_n(s,i)-P_{n+1}(s,i),\qq \L_n(s,i)\deq\Th_{n-1}(s,i)-\Th_n(s,i),\qq n\ges1.$$
    Then for $n\ges1$, we have
    \begin{align}
      \label{Di-equa1}
      -\dot \D_n(s,i)=&\dot{P}_{n+1}(s,i)-\dot{P}_n(s,i) \nonumber\\
      =&P_n(s,i)A_{n-1}(s,i)+A_{n-1}(s,i)^\top P_n(s,i)+C_{n-1}(s,i)^\top P_n(s,i)C_{n-1}(s,i)\nonumber\\
                      &+\Th_{n-1}(s,i)^\top R(s,i)\Th_{n-1}(s,i)+S(s,i)^\top\Th_{n-1}(s,i)+\Th_{n-1}(s,i)^\top S(s,i)\nonumber\\
                      &-P_{n+1}(s,i)A_n(s,i)-A_n(s,i)^\top P_{n+1}(s,i)-C_n(s,i)^\top P_{n+1}(s,i)C_n(s,i)\nonumber\\
                      &-\Th_n(s,i)^\top R(s,i)\Th_n(s,i)-S(s,i)^\top\Th_n(s,i)-\Th_n(s,i)^\top S(s,i)+\sum_{k=1}^D\l_{ik}(s)\D_n(s,k)\\
      =&\D_n(s,i)A_n(s,i)+A_n(s,i)^\top\D_n(s,i)+C_n(s,i)^\top\D_n(s,i)C_n(s,i)\nonumber\\
                      &+P_n(s,i)(A_{n-1}(s,i)-A_n(s,i))+(A_{n-1}(s,i)-A_n(s,i))^\top P_n(s,i)\nonumber\\
                      &+C_{n-1}(s,i)^\top P_n(s,i)C_{n-1}(s,i)-C_n(s,i)^\top P_n(s,i)C_n(s,i)\nonumber\\
                      &+\Th_{n-1}(s,i)^\top R(s,i)\Th_{n-1}(s,i)-\Th_n(s,i)^\top R(s,i)\Th_n(s,i)\nonumber\\
                      &+S(s,i)^\top\L_n(s,i)+\L_n(s,i)^\top S(s,i)+\sum_{k=1}^D\l_{ik}(s)\D_n(s,k).\nonumber
    \end{align}
    By (\ref{Iteration-i}), we have the following:
    \begin{align}
      \label{Di-equa2}\left\{\2n\ba{ll}
      \ns\ds A_{n-1}(s,i)-A_n(s,i)=B(s,i)\L_n(s,i),\\
      C_{n-1}(s,i)-C_n(s,i)=D(s,i)\L_n(s,i),\\
      \ns\ds C_{n-1}(s,i)^\top P_n(s,i)C_{n-1}(s,i)-C_n(s,i)^\top P_n(s,i)C_n(s,i)\\
      =\L_n(s,i)^\top D(s,i)^\top P_n(s,i)D(s,i)\L_n(s,i)+C_n(s,i)^\top P_n(s,i)D(s,i)\L_n(s,i)\\
      \q+\L_n(s,i)^\top D(s,i)^\top P_n(s,i)C_n(s,i),\\
      \ns\ds \Th_{n-1}(s,i)^\top R(s,i)\Th_{n-1}(s,i)-\Th_n(s,i)^\top R(s,i)\Th_n(s,i)\\
      =\L_n(s,i)^\top R(s,i)\L_n(s,i)+\L_n(s,i)^\top R(s,i)\Th_n(s,i)+\Th_n(s,i)^\top R(s,i)\L_n(s,i).\ea\right.
    \end{align}
    Note that
    \begin{align*}
      &B(s,i)^\top P_n(s,i)+D(s,i)^\top P_n(s,i)C_n(s,i)+R(s,i)\Th_n(s,i)+S(s,i)\\
      &=B(s,i)^\top P_n(s,i)+D(s,i)^\top P_n(s,i)C(s,i)+S(s,i)+(R(s,i)+D(s,i)^\top P_n(s,i)D(s,i))\Th_n(s,i)=0.
    \end{align*}
    Thus, plugging (\ref{Di-equa2}) into (\ref{Di-equa1}) yields
    \begin{align}
      \label{Di-equa3}\ba{ll}
      \ns\ds&-\,\big[\dot\D_n(s,i)+\Delta_n(s,i)A_n(s,i)+A_n(s,i)^\top\D_n(s,i)+C_n(s,i)^\top\D_n(s,i)C_n(s,i)+\sum_{k=1}^D\l_{ik}(s)\D_n(s,k)\big]\\
      \ns\ds&=P_n(s,i)B(s,i)\L_n(s,i)+\L_n(s,i)^\top B(s,i)^\top P_n(s,i)+\L_n(s,i)^\top D(s,i)^\top P_n(s,i)D(s,i)\L_n(s,i)\\
            &\q+C_n(s,i)^\top P_n(s,i)D(s,i)\L_n(s,i)+\L_n(s,i)^\top D(s,i)^\top P_n(s,i)C_n(s,i)+\L_n(s,i)^\top R(s,i)\L_n(s,i)\\
      \ns\ds&\q+\L_n(s,i)^\top R(s,i)\Th_n(s,i)+\Th_n(s,i)^\top R(s,i)\L_n(s,i)+S(s,i)^\top\L_n(s,i)+\L_n(s,i)^\top S(s,i)\\
      \ns\ds&=\L_n(s,i)^\top\big[R(s,i)+D(s,i)^\top P_n(s,i)D(s,i)\big]\L_n(s,i)\\
            &\q+\big[P_n(s,i)B(s,i)+C_n(s,i)^\top P_n(s,i)D(s,i)+\Th_n(s,i)^\top R(s,i)+S(s,i)^\top\big]\L_n(s,i)\\
            &\q+\L_n(s,i)^\top\big[B(s,i)^\top P_n(s,i)+D(s,i)^\top P_n(s,i)C_n(s,i)+R(s,i)\Th_n(s,i)+S(s,i)\big]\\
      \ns\ds&=\L_n(s,i)^\top\big[R(s,i)+D(s,i)^\top P_n(s,i)D(s,i)\big]\L_n(s,i)\ges0.\ea
    \end{align}
    Noting that $\D_n(T,i)=0$ and using Proposition \ref{sec:F-K}, also noting (\ref{R+Pi-lowerbound}), we obtain
    $$P_1(s,i)\ges P_n(s,i)\ges P_{n+1}(s,i)\ges\a I,\qq\forall s\in [0,T],\q\forall n\ges1.$$
    Therefore, the sequence $\{P_n(s,i)\}_{n=1}^\i$ is uniformly bounded.
    Consequently, there exists a constant $K>0$ such that (noting
    (\ref{R+Pi-lowerbound}))
    \begin{align}
      \label{Di-equa4}\left\{\2n\ba{ll}
      \ns\ds|P_n(s,i)|,\ |R_n(s,i)|\les K,\\
      \ns\ds|\Th_n(s,i)|\les K\big(|B(s,i)|+|C(s,i)|+|S(s,i)|\big),\\
      \ns\ds|A_n(s,i)|\les |A(s,i)|+K|B(s,i)|\big(|B(s,i)|+|C(s,i)|+|S(s,i)|\big),\\
      \ns\ds|C_n(s,i)|\les |C(s,i)|+K\big(|B(s,i)|+|C(s,i)|+|S(s,i)|\big),\ea\right.
      \ae s\in [0,T],\forall i\in\cS, \forall n\ges0,
    \end{align}
    where $R_n(s,i)\deq R(s,i)+D^\top(s,i)P_n(s,i)D(s,i)$. Observe that
    \begin{align}
      \label{Di-equa5}
      \L_n(s,i)=&\Th_{n-1}(s,i)-\Th_n(s,i) \nonumber\\
      =&R_n(s,i)^{-1}D(s,i)^\top\D_{n-1}(s,i)D(s,i)R_{n-1}(s,i)^{-1}\hat S_n(s,i)\\
                &-R_{n-1}(s,i)^{-1}\big[B(s,i)^\top\D_{n-1}(s,i)+D(s,i)^\top\D_{n-1}(s,i)C(s,i)\big]. \nonumber
    \end{align}
    where $\hat S_n(s,i):=B(s,i)^\top P_n(s,i)+D(s,i)^\top P_n(s,i)C(s,i)+S(s,i)$.
    Thus, noting (\ref{Di-equa4}), one has
    \begin{eqnarray}
      \label{3.22}\ba{ll}
      \ns\ds|\L_n(s,i)^\top R_n(s,i)\L_n(s,i)|\les\(|\Th_n(s,i)|+|\Th_{n-1}(s,i)|\)\,|R_n(s,i)|\,
      |\Th_{n-1}(s,i)-\Th_n(s,i)|\\
      \ns\ds\qq\qq\qq\qq\qq\q\2n~\les K\(|B(s,i)|+|C(s,i)|+|S(s,i)|\)^2|\D_{n-1}(s,i)|.\ea
    \end{eqnarray}
    Equation (\ref{Di-equa3}), together with $\D_n(T,i)=0$, implies that
    \begin{align*}
      \D_n(s,i)=\int^T_s\big[&\D_n(r,i)A_n(r,i)+A_n(r,i)^\top\D_n(r,i)+C_n(r,i)^\top\D_n(r,i)C_n(r,i)\\
                             &+\L_n(r,i)^\top R_n(r,i)\L_n(r,i)+\sum_{k=1}^D\l_{ik}(r)\D_n(r,k)\big]dr.
    \end{align*}
    \todo{In order to make use of Gronwall's inequaltiy, I revised the method and please to check if the following process is correct.}
    Making use of (\ref{3.22}) and still noting (\ref{Di-equa4}), we get
    \begin{align}
      \label{eq:D_nineq}
      |\D_n(s,i)|\les \int^T_s\f(r)\[\Big\vert\sum_{k=1}^D\D_n(r,k)\Big\vert+|\D_{n-1}(r,i)|\]dr,\qq\forall s\in[0,T],\q\forall n\ges1,
    \end{align}
    where $\f(\cd)$ is a nonnegative integrable function independent of $\D_n(\cd, \cd)$. Let
    \begin{align*}
      \Vert\D_n(s)\Vert:=\max_{k=1}^D|\D_n(s,k)|.
    \end{align*}
    Thus from (\ref{eq:D_nineq}), we have
    \begin{align}
      \Vert\D_n(s)\Vert\les \int^T_s\f(r)\[\Vert\D_n(r)\Vert+\Vert\D_{n-1}(r)\Vert\]dr,\qq\forall s\in[0,T],\q\forall n\ges1,
    \end{align}
    By Gronwall's inequality,
    $$\Vert\D_n(s)\Vert\les e^{\int_0^T\f(r)dr}\int^T_s\f(r)\Vert\D_{n-1}(r)\Vert dr\equiv c\int^T_s\f(r)\Vert\D_{n-1}(r) .$$
    Set
    $$a\deq\max_{0\les s\les T}\Vert\D_0(s)\Vert.$$
    By induction, we deduce that
    $$||\D_n(s)||\les a{c^n\over n!}\(\int_s^T\f(r)dr\)^n,\qq\forall s\in[0,T],$$
    which implies the uniform convergence of $\{P_n(\cd,\cd)\}_{n=1}^\i$. We denote $P(\cd,\cd)$ the limit of $\{P_n(\cd,\cd)\}_{n=1}^\infty$, then
    (noting (\ref{R+Pi-lowerbound}))
    $$R(s,i)+D(s,i)^\top P(s,i)D(s,i)=\lim_{n\to\i}R(s,i)+D(s,i)^\top P_n(s,i)D(s,i)\ges\eps I,
    \qq\ae~s\in[0,T],$$
    and as $n\to\infty$,
    $$\left\{\2n\ba{ll}
      \ns\ds\Th_n(s,i)\to-\hat R(s,i)\hat S(s,i)\equiv\Th(s)  & \hb{in $L^2$},\\
      \ns\ds A_n(s,i)\to A(s,i)+B(s,i)\Th(s)  & \hb{in $L^1$},\\
      \ns\ds  C_n(s,i)\to C(s,i)+D(s,i)\Th(s)  & \hb{in $L^2$}.\ea\right.$$
    Therefore, $P(\cd,\cd)$ satisfies the following equation:
    $$\left\{\2n\ba{ll}
      \ns\ds\dot P(s,i)+P(s,i)\big[A(s,i)+B(s,i)\Th(s)\big]+\big[A(s,i)+B(s,i)\Th(s)\big]^\top P(s,i)\\
      \ns\ds \q+\big[C(s,i)+D(s,i)\Th(s)\big]^\top P(s,i)\big[C(s,i)+D(s,i)\Th(s)\big]+\Th(s)^\top R(s,i)\Th(s)\\
      \q+S(s,i)^\top\Th(s)+\Th(s)^\top S(s,i)+Q(s,i)+\sum_{k=1}^D\l_{ik}(s)P(s,k)=0,\qq\ae~s\in[0,T],\\
      \ns\ds P(T,i)=G(T,i),\ea\right.$$
    which is equivalent to (\ref{Riccati}).

    \ms

    (ii) $\Ra$ (i). Let $P(\cd,\cd)$ be the strongly regular solution of (\ref{Riccati}).
    Then there exists a $\eps>0$ such that
    \bel{iitoi}\hat R(s,i) \ges \eps  I,\qq\ae~s\in[0,T].\ee
    Set
    $$\Th(s)\deq -\hat R(s,\a(s))\hat S(s,\a(s))\in L^2(0,T;\dbR^{m\times n}).$$
    For any $u(\cd)\in\cU[0,T]$, let $X_0^u(\cd\,;t,0,i)$ be the solution of
    $$\left\{\2n\ba{ll}
      \ns\ds dX_0^u(s;t,0,i)=\big[A(s,\a(s))X_0^u(s;t,0,i)+B(s,\a(s))u(s)\big]ds\\
      \qq\qq\q+\big[C(s,\a(s))X_0^u(s;t,0,i)+D(s,\a(s))u(s)\big]dW(s),\qq s\in[t,T], \\
      \ns\ds\hspace{0.2cm} X_0^{0,u}(t)=0.\ea\right.$$
    Applying It\^o's formula to $s\mapsto\langle P(s,\a(s))X_0^u(s;t,0,i),X_0^u(s;t,0,i)\rangle$, we have
    \begin{align*}
      &J^0(t,0;u(\cd))\\
      =&\dbE\Bigg\{\Blan G(T,\a(T))X_0^u(T;t,0,i),X_0^u(T;t,0,i)\Bran +
                               \int_t^T\bigg[\Blan Q(s,\a(s))X_0^u(s;t,0,i), X_0^u(s;t,0,i)\Bran\\
                             &\qq\q+2\Blan\1nS(s,\a(s))X_0^u(s;t,0,i),u(s)\Bran+\Blan R(s,\a(s))u(s),u(s)\Bran\bigg] ds\Bigg\}\\
      =&\dbE\int_t^T\[\lan\dot P(s,\a(s))X_0^u(s;t,0,i),X_0^u(s;t,0,i)\ran\\
                             &\qq\q+\lan P(s,\a(s))\big[A(s,\a(s))X_0^u(s;t,0,i)+B(s,\a(s))u(s)\big],X_0^u(s;t,0,i)\ran\\
                             &\qq\q+\lan P(s,\a(s))X_0^u(s;t,0,i),A(s,\a(s))X_0^u(s;t,0,i)+B(s,\a(s))u(s)\ran\\
                             &\qq\q+\lan P(s,\a(s))\big[C(s,\a(s))X_0^u(s;t,0,i)+D(s,\a(s))u(s)\big],C(s,\a(s))X_0^u(s;t,0,i)+D(s,\a(s))u(s)\ran\\
                             &\qq\q+\lan Q(s,\a(s))X_0^u(s;t,0,i),X_0^u(s;t,0,i)\ran+2\lan S(s,\a(s))X_0^u(s;t,0,i),u(s)\ran\\
                             &\qq\q+\lan R(s,\a(s))u(s),u(s)\ran+\lan \sum_{k=1}^D\l_{\a(s-),k}(s)P(s,k)X_0^u(s;t,0,i),X_0^u(s;t,0,i)\ran\]ds\\
      =&\dbE\int_t^T\[\lan\h Q(s,\a(s))X_0^u(s;t,0,i),X_0^u(s;t,0,i)\ran+2\lan\h S(s,\a(s))X_0^u(s;t,0,i),u(s)\ran+\lan\h R(s,\a(s))u(s),u(s)\ran\]ds\\
      =&\dbE\int_t^T\big[\lan\Th(s)^\top\h R(s,\a(s))\Th(s) X_0^u(s;t,0,i),X_0^u(s;t,0,i)\ran
      \\
           &\qq\q-2\lan\h R(s,\a(s))\Th(s) X_0^u(s;t,0,i),u(s)\ran+\lan\h R(s,\a(s))u(s),u(s)\ran\big]ds\\
      =&\dbE\int_t^T\lan\big[\hat R(s,\a(s))\big[u(s)-\Th(s) X_0^u(s;t,0,i)\big],u(s)-\Th(s) X_0^u(s;t,0,i)\ran ds.
    \end{align*}
    Noting (\ref{iitoi}) and making use of Lemma \ref{uniformconvex}, we obtain that
    \begin{align*}
      J^0(t,0;u(\cd))=&\dbE\int_t^T\lan \hat R(s,\a(s))\big[u(s)-\Th(s) X_0^u(s;t,0,i)\big],u(s)-\Th(s) X_0^u(s;t,0,i)\ran ds\\
      \ges&\l\g\dbE\int_t^T|u(s)|^2ds,
            \q\forall u(\cd)\in\cU[t,T],
    \end{align*}
    for some $\g>0$. Then (i) holds.
  \end{proof}

  \begin{rmk}
    From the first part of the proof of Theorem 4.6, we see that if (\ref{J>l}) holds,
    then the strongly regular solution of (\ref{Riccati}) satisfies (\ref{strong-regular})
    with the same constant $\l>0$.
  \end{rmk}

  Combining Theorem \ref{sec:closedlooplyapunov} and Theorem \ref{sec:unifconv-strongregusolu}, we obtain the following corollary.

  \begin{coro}
    \label{sec:opt-open-cont-u}
    Let $P(\cd,\cd)$ be the unique strongly regular solution of {\rm(\ref{Riccati})}
    with $(\eta(\cd),\z(\cd))$ being the adapted
    solution of {\rm(\ref{eta-zeta-xi})}. $\hat R(\cd,\cd)$ and $\hat \rho(\cd,\cd)$ are defined by \eqref{eq:hatsr} and \eqref{eq:hatrrho} respectively. Suppose that {\rm(H1)--(H2)} and {\rm(\ref{J>l})}
    hold. Then Problem {\rm(M-SLQ)} is uniquely open-loop solvable at any
    $(t,x,i)\in[0,T)\times\dbR^n\times \cS$ with the open-loop optimal control
    $u^*(\cd)$ being of a state feedback form:
    \begin{align}
      \label{opti-biaoshi} u^*(\cd)=-\hat R(\cd,\a(\cd))^{-1}\hat S(\cd,\a(\cd))X^*(\cd)-\hat R(\cd,\a(\cd))^{-1}\hat\rho(\cd,\a(\cd))
    \end{align}
   where $X^*(\cd)$ is  the solution
    of the following closed-loop system:
    \begin{align}
      \label{eclosed-loop-state}\left\{\2n\ba{ll}
      \ns\ds dX^*(s)=\Big\{\big[A(s,\a(s))-B(s,\a(s))\hat R(s,\a(s))^{-1}\hat S(s,\a(s))\big]X^*(s)\\
      \ns\ds\qq\qq\q~-B(s,\a(s))\hat R(s,\a(s))^{-1}\hat \rho(s,\a(s))+b(s,\a(s))\Big\}ds\\
      \ns\ds\qq\qq~~\1n+\Big\{\big[C(s,\a(s))-D(s,\a(s))\hat R(s,\a(s))^{-1}\hat S(s,\a(s))\big]X^*\\
      \ns\ds\qq\qq\qq~-D(s,\a(s))\hat R(s,\a(s))^{-1}\hat\rho(s,\a(s))+\si(s,\a(s))\Big\}dW(s),\qq s\in[t,T], \\
      \ns\ds X^*(t)=x.\ea\right.
    \end{align}
     %
    % Moreover, the value function is given by
    %%
    % \bel{}\ba{ll}
    %%
    % \ns\ds V(t,x)=\dbE\bigg\{\langle P(t)x,x\rangle+2\langle\eta(t),x\rangle+\int_t^T\[\langle
    %   P\si,\si\rangle+2\langle\eta,b\rangle+2\langle\z,\si\rangle\\
    %%
    %   \ns\ds\qq\qq\q\ ~-\lan(R+D^\top PD)^\dag(B^\top\eta+D^\top\z+D^\top P\si+\rho),
    %   B^\top\eta+D^\top\z+D^\top P\si+\rho\ran\]ds\bigg\}.\ea\    \ms
  \end{coro}

  \begin{proof}
    By Theorem \ref{sec:unifconv-strongregusolu}, the Riccati equation (\ref{Riccati})
    admits a unique strongly regular solution $P(\cd,\cd)\in C([0,T]\times\cS;\dbS^n)$.
    Hence, the adapted solution $(\eta(\cd),\z(\cd))$ of {\rm(\ref{eta-zeta-xi})}
    satisfies (\ref{eta-zeta-regularity}) automatically. Now applying Theorem \ref{sec:closedloop-regusolu}
    and noting the remark right after Definition \ref{sec:defnofopen-closeloop}, we get the desired result.
  \end{proof}
  \begin{rmk}
    Under the assumptions of Corollary \ref{sec:opt-open-cont-u}, when $b(\cd,\cd), \si(\cd,\cd), g(\cd,\cd), q(\cd,\cd), \rho(\cd,\cd)=0$,
    the adapted solution of (\ref{eta-zeta-xi}) is $(\eta(\cd),\z(\cd))\equiv(0,0)$. Thus, for Problem ${\rm(M-SLQ)}^0$,
    the unique optimal control $u^*(\cd)$ at initial pair $(t,x)\in[0,T)\times\dbR^n$ is given by
    \begin{align}
      \label{opti-biaoshi-0} u^*(\cd)=-\hat R(\cd,\a(\cd))^{-1}\hat S(\cd,\a(\cd))X^*(\cd),
    \end{align}
      with $P(\cd,\cd)$ being the unique strongly regular solution of {\rm(\ref{Riccati})}
      and $X^*(\cd)$ being the solution of the following closed-loop system:
      \begin{align}
        \label{closed-loop-state-0}\left\{\2n\ba{ll}
        \ns\ds dX^*(s)=\big[A(s,\a(s))-B(s,\a(s))\hat R(s,\a(s))^{-1}\hat S(s,\a(s))\big]X^*(s)ds\\
        \ns\ds\qq\qq~~\1n+\big[C(s,\a(s))-D(s,\a(s))\hat R(s,\a(s))^{-1}\hat S(s,\a(s))\big]X^*(s)dW(s),\qq s\in[t,T], \\
      \ns\ds X^*(t)=x.\ea\right.
    \end{align}
    Moreover, by (\ref{Value}), the value function of Problem ${\rm(M-SLQ)}^0$ is given by
    \bel{} V^0(t,x,i)=\langle P(t,i)x,x\rangle, \qq (t,x,i)\in[0,T]\times\dbR^n\times \cS.\ee
  \end{rmk}

  % \begin{rmk}
  %   Combing the above results, we have the following relationships:

  %   {\footnotesize
  %     $$\ba{ll}
  %     %
  %     % \ns\ds\qq\qq\qq\qq\qq\q\boxed{G\ges0,\q R\gg0,\q Q-S^\top R^{-1}S\ges0}\\
  %     %
  %     % \ns\ds\qq\qq\qq\qq\qq\qq\q~\Downarrow\qq\qq\qq\qq\qq\Downarrow\\
  %     %
  %     \ns\ds \boxed{\hb{$u(\cd)\mapsto J^0(t,0;u(\cd))$ uniformly
  %         convex}}\Longleftrightarrow\,\boxed{\sc\hb{RE strongly regularly solvable}}
  %     \qq\,\Rightarrow\qq\,\boxed{\sc\hb{RE regularly solvable}}\\
  %     %
  %     \ns\ds\qq\qq\qq\qq\,\q\qq\qq\qq\qq\qq\qq\qq\qq\qq\q\,\Downarrow\qq\qq\qq\qq\qq\qq\qq\qq\q~\Updownarrow\\
  %     %
  %     \ns\ds~\,\boxed{\hb{(M-SLQ) uniquely open-loop
  %         solvable}}\q\Leftarrow\q\boxed{\hb{(M-SLQ) uniquely closed-loop
  %         solvable}}\q\Rightarrow\q\boxed{\hb{$\rm(M-SLQ)^0$ closed-loop
  %         solvable}}\ea$$}
  %   %

  % \end{rmk}

  % \small

 %  \bibliography{/Users/Alsichkann/Sync/Dropbox/Mybib}
 %  \bibliographystyle{plainnat}

% \bibliography{/Users/Shared/Sync/Dropbox/Mybib}
 % \bibliographystyle{/Users/Shared/Sync/Dropbox/plainnat-doi}

  %  \bibliography{/Users/Alsichkann/Sync/Dropbox/Doi}
 % \bibliographystyle{/Users/Alsichkann/Sync/Dropbox/plainnat-doi}

\end{document}